\documentclass[11pt]{article}
\usepackage[margin =1in]{geometry}
\usepackage{amssymb,amsthm}
\usepackage{amsmath}
\usepackage{enumerate}
\usepackage{hyperref}
\usepackage{cite}
\usepackage{color}

\usepackage{mathtools}

\numberwithin{equation}{section}

\newtheorem{thm}{Theorem}[section]
\newtheorem{definition}[thm]{Definition}

\newtheorem{lem}[thm]{Lemma}
\newtheorem{cor}[thm]{Corollary}

\newtheorem{assumption}[thm]{Assumption}

\newtheorem{remark}{Remark}

\theoremstyle{remark}

\newcommand{\ip}[1] {\langle #1 \rangle } 
\newcommand{\norm}[1] {\left \| #1 \right \|} 
\newcommand{\Real}{{\mathbb R}} 

\newcommand{\Prob} {{\bf{Pr}}} 
\newcommand{\Exp}{{\bf{E}}} 

\newcommand{\Stepsize}[2]{\mathcal{A}_{#1}^{#2}} 
\newcommand{\stepsize}[2]{\alpha_{#1}^{#2}}
\newcommand{\pf}{p_f} 
\newcommand{\pg}{p_g} 
\newcommand{\varConstF}{\kappa_f} 
\newcommand{\constF}{\varepsilon_f} 
\newcommand{\constG}{\kappa_g} 
\newcommand{\Control}[2]{\Delta_{#1}^{#2}} 
\newcommand{\control}[2]{\delta_{#1}^{#2}} 

\newcommand{\success}{1_{\text{Succ}}} 
\newcommand{\notsuccess}{1_{\text{Succ}^c}} 
\newcommand{\True}{1_{\text{R}}} 
\newcommand{\notTrue}{1_{\text{U}}} 
\newcommand{\goodModel}{I_k}
\newcommand{\goodFunction}{J_k}

\newcommand{\pastone}{\mathcal{F}_{k-1}^{G \cdot F}}
\newcommand{\pasthalf}{\mathcal{F}_{k-1/2}^{G \cdot F}}

\usepackage{mathtools}

\usepackage[boxruled]{algorithm2e}
\title{A Stochastic Line Search Method with Convergence Rate Analysis}
\author{ Courtney Paquette\thanks{Department of Industrial and Systems Engineering, Lehigh University,
Harold S. Mohler Laboratory, 200 West Packer Avenue, Bethlehem, PA 18015-1582, USA.
 {\tt cop318@lehigh.edu}.
 The work of this author was partially supported by NSF TRIPODS Grant 17-40796 and DMS 18-03289.}
\and Katya Scheinberg\thanks{Department of Industrial and Systems Engineering, Lehigh University,
Harold S. Mohler Laboratory, 200 West Packer Avenue, Bethlehem, PA 18015-1582, USA.
{\tt katyas@lehigh.edu}. The work of this author was  partially supported by NSF Grants CCF 16-18717 and TRIPODS 17-40796,
and DARPA Lagrange award HR-001117S0039. } 
}

\begin{document}

\maketitle

\begin{abstract} For deterministic optimization, line-search methods
  augment algorithms by providing stability and improved efficiency. We
  adapt a classical backtracking Armijo line-search to the
stochastic optimization setting. While traditional line-search relies on exact computations of the gradient and values of the objective function,
our method assumes that these values are available up to some
dynamically adjusted accuracy which holds with some sufficiently
large, but fixed, probability. We show the
expected number of iterations to reach a near stationary point matches
the worst-case efficiency of typical first-order methods, while for
convex and strongly convex objective, it achieves rates
of deterministic gradient descent in function values. 
\end{abstract}

\section{Introduction}
In this paper we consider the classical  stochastic optimization problem
\begin{align}
\min_{x \in \Real^n} \left\{ f(x) = \mathbb{E} [ \tilde f(x;\xi) ] \right\},  \label{main_prob_expected_risk}
\end{align}
where $\xi$ is a random variable  obeying some distribution. In the case of empirical risk minimization with a finite training set, $\xi_i$ is a random variable that is defined
 by a single random sample drawn uniformly from the training set.
More generally $\xi$ may represents a sample or a set of samples
drawn from the data distribution. 

The most widely used method to solve  \eqref{main_prob_expected_risk} is the stochastic gradient descent (SGD) \cite{RM1951}. Due to its low iteration cost, SGD is often preferred to the standard gradient descent
(GD) method for empirical risk minimization. Despite the prevalent use of SGD, it has known challenges and inefficiencies. First, the direction may not represent a descent direction, and second, the method is sensitive to the step-size (learning rate) which is often poorly overestimated. Various authors have attempted to address this last issue, see \cite{adaGrad,George2006,pmlr-v28-hennig13,adam}. Motivated by these facts, we turn to the deterministic optimization approach for adaptively selecting step sizes - GD with Armijo back-tracking line-search. 

\paragraph{Related work.} In general, GD with back-tracking requires computing a full gradient \textit{and} function evaluation - too expensive of an operation for the general problem \eqref{main_prob_expected_risk}. On the other hand, the per-iteration convergence rate for GD is superior to SGD making it an attractive alternative.
Several works have attempted to transfer ideas from deterministic GD  to the stochastic setting with the intent of diminishing the gradient computation, by using dynamic gradient sampling, e.g. \cite{RByrd_etal_2012, Schmidt,   2014pasglyetal}. 
However, these works address only convex setting. Moreover for them to obtain convergence rates matching those of GD in expectation, a small constant step-size must be known in advance and the sample size needs to be increased at a pre-described rate thus decreasing the variance of gradient estimates.  
Recently, in \cite{Bollapragada2017} an adaptive sample size selection strategy was proposed where sample size is selected based on the reduction of the gradient (and not pre-described). For convergence rates to be derived, however, an assumption has to be made that these sample sizes can be selected based on the size of the true gradient, which is, of course, unknown. 
In \cite{Tripuraneni2017} a second-order method that subsamples gradient and Hessian is proposed, however, the sample size is simply assumed to be sufficiently large, so that essentially, the method behaves as a deterministic inexact method with high probability. 

In  \cite{Bollapragada2017} and \cite{Schmidt} a practical back-tracking line search is proposed, combined with the their sample size selection. In both cases the backtracking is based on Armijo line search condition applied to function estimates that are computed on the same batch as the gradient estimates and is essentially a heuristic. A very different type of line-search based on probabilistic Wolfe condition is proposed in \cite{Prob_line_search}, however, it aims at improving step size selection for SGD and has no theoretical guarantees. 

\paragraph{Our contribution.}  In this work we propose an adaptive backtracking line-search method, where the sample sizes for  gradient and function estimates are  chosen \textit{adaptively} using \textit{knowable quantities} along with the step-size. We show that this method converges to the optimal solution \textit{ with probability one} and derive \textit{strong convergence rates} that match those of the deterministic gradient descent methods in the nonconvex $O(\varepsilon^{-2})$, convex $O(\varepsilon^{-1})$, and strongly convex $O(\log(\varepsilon^{-1}))$ cases.  This paper offers the first stochastic line search method with convergence rates analysis, and is the first to provide convergence rates analysis
for adaptive sample size selection   based on \textit{knowable quantities}.   

\paragraph{Background.} There are many types of (deterministic) line-search
methods, see \cite[Chapter 3]{Nocedal_Wright}, but all share a common
philosophy. First, at each iteration, the method computes a search direction
$d_k$ by \textit{e.g.} the gradient or (quasi) Newton
directions. Next, they determine how far to move in the direction
through the univariate function, $\phi(\alpha) = f(x_k +
  \alpha d_k)$, to find the stepsize $\alpha_k$. Typical
  line-searches try out a sequences
  of potential values for the stepsize, accepting $\alpha$ once some
  verifiable criteria becomes satisfied. One popular line-search criteria specifies
an acceptable step length should give \textit{sufficient decrease} in
the objective function $f$:
\begin{equation} \label{eq: armijo} \text{(Armijo condition \cite{armijo})} \qquad f(x_k+\alpha d_k) \le f(x_k) -
  \theta \alpha \norm{\nabla f(x_k)}^2,\end{equation}
where the constant $\theta \in (0,1)$ is chosen by the user and $d_k =
-\nabla f(x_k)$. Larger
step sizes imply larger gains towards optimality and lead to fewer
overall iterations. When step sizes get too small or worse $0$, no
progress is made and the algorithm stagnates. A popular way to
systematically search the domain of $\alpha$ while simultaneously
preventing small step sizes is backtracking. Backtracking starts with
an overestimate of $\alpha$ and
decreases it until \eqref{eq: armijo} becomes true. Our
exposition is on a stochastic version of backtracking using the
stochastic gradient estimate as a search direction and stochastic function estimates in \eqref{eq: armijo}. In the remainder of the paper, all random quantities will 
be denoted by capitalized letters and their respective realizations by corresponding lower case letters.

\section{Stochastic back-tracking line search method}

We present here our main algorithm for GD with back-tracking line
search. We impose the standard assumption on the objective function. 
\begin{assumption}\label{assumpt: objective_function} \rm{We assume
    that all iterates $x_k$ of Algorithm~\ref{alg:line_search_method}
    satisfy $x_k \in \Omega$ where $\Omega$ is a set in
  $\mathbb{R}^n$. Moreover, the gradient of $f$ is $L$-Lipschitz continuous for all $x \in \Omega$ and that 
\[f_{\min} \le f(x), \qquad \text{for all $x \in \Omega$}.\]
}
\end{assumption}

\subsection{Outline of method} At each iteration, our scheme computes
a random direction $g_k$ via \textit{e.g.}  a minibatch stochastic
gradient estimate or sampling the function $f(x)$ itself and using finite
differences. Then, we compute stochastic function estimates at the current
iterate and prospective new iterate, resp. $f_k^0$ and $f_k^s$. We check the Armijo condition \cite{armijo} using the \textit{stochastic estimates}
\begin{equation} \label{eq: armijo_stoch} \text{(Stochastic) Armijo}
  \qquad f_k^s \le f_k^0 - \theta \stepsize{k}{} \norm{g_k}^2. \end{equation}
If \eqref{eq: armijo_stoch} holds, the next iterate becomes $x_{k+1} = x_k -
\stepsize{k}{} g_k$ and stepsize $\stepsize{k}{}$ increases; otherwise
$x_{k+1}=x_k$ and $\stepsize{k}{}$ decreases, as is typical in
(deterministic) back-tracking line searches.  

Algorithm \ref{alg:line_search_method} describes our method.\footnote{We state the algorithm using the lower case notation to  represent a realization of the algorithm} Unlike classical back-tracking line search, there is an additional
control, $\control{k}{}$, which serves as a guess of the true function decrease
and controls the accuracy of the function estimates. We discuss this further next. 
	\begin{algorithm}[t!]
		\label{alg:line_search_method}
          \textbf{Initialization:} Choose constants $\gamma >1$,
          $\theta \in (0,1)$ and $\stepsize{\max}{}$. Pick
          initial point $x_0$, $\stepsize{0}{} =\gamma^{j_0} \stepsize{\max}{}$ for some $j_0\leq0$, and $\control{0}{}$. \\
\textit{Repeat for $k = 0, 1, \hdots $}
\begin{enumerate}
\begin{minipage}{0.85 \textwidth} \item \textbf{Compute a gradient estimate} Based on $\stepsize{k}{} $ compute a gradient estimate $g_k$. Set the step $s_k = -\stepsize{k}{} g_k$ \end{minipage}
\begin{minipage}{0.85 \textwidth} \item \textbf{Compute function estimates} Based on $\control{k}{}$, $g_k$ and $\stepsize{k}{} $ obtain estimates of $f_k^0$ and
  $f_k^s$ of $f(x_k)$ and $f(x_k+s_k)$ respectively. \end{minipage}
\item \textbf{Check sufficient decrease}\\
Check if
\begin{minipage}{0.75 \textwidth} \begin{equation} \label{eq:
      sufficient_decrease} f_k^s \le f_k^0 -\stepsize{k}{} \theta
    \norm{g_k}^2. \end{equation} \end{minipage}
\item \textbf{Successful step}\\
If \eqref{eq: sufficient_decrease} set $x_{k+1} = x_k-\stepsize{k}{}
g_k$ and $\stepsize{k+1}{} = \min\{ \stepsize{\max}{},
\gamma\stepsize{k}{}\}$.
\begin{itemize}
\item \textbf{Reliable step}: If $\stepsize{k}{} \norm{g_k}^2 \ge
    \control{k}{2}$, then increase $\control{k+1}{2} = \gamma
    \control{k}{2}$. 
\item \textbf{Unreliable step}: If $\stepsize{k}{} \norm{g_k}^2 <
  \control{k}{2}$, then decrease $\control{k+1}{2} = \gamma^{-1}\control{k}{2}$.
\end{itemize}
\item \textbf{Unsuccessful step}\\
Otherwise, set $x_{k+1} = x_k$, $\stepsize{k+1}{} = \gamma^{-1}
\stepsize{k}{}$, and $\control{k+1}{2} = \gamma^{-1} \control{k}{2}$. 
\end{enumerate}
Let $k = k+1$. 
		\caption{Line search method}
	\end{algorithm}
	\bigskip

\paragraph{Challenges with randomized line-search.} Due to the
stochasticity of the gradient and/or function values, two major
challenges result:
\begin{itemize}
\item a series of erroneous unsuccessful steps cause
  $\Stepsize{k}{}$ to become arbitrarily small;
\item steps may falsely satisfy \eqref{eq: armijo_stoch} leading to
  objective value at the next iteration arbitrarily larger than the current
  iterate.
\end{itemize}
Convergence proofs for deterministic line searches rely on the fact that neither
of the above problems arise. Our approach
controls the probability with which the random gradients and function values are representative of their
true counterparts. When this probability is large enough, the method tends to make successful steps when $\Stepsize{k}{}$ is sufficiently small, hence $\Stepsize{k}{}$ behaves like a random walk with an upward drift  thus staying away from $0$. 

Yet, even when the
probability of good gradients/function estimates is near 1, it is not
guaranteed that $\Exp(f(X_{k+1})|X_{k}) < f(X_k)$ holds at each
iteration due to the second issue - possible arbitrary increase of the objective. 
Since random gradient may not be representative of the true 
gradient the  function estimate accuracy and thus the 
 expected improvement needs to be  controlled by a different quantity, $\Control{k}{2}$. 
When the predicted decrease in the true function
matches the expected function estimate accuracy  ($\Control{k}{2} \le
\Stepsize{k}{}\norm{G_k}^2$), we call the step reliable and increase
the parameter $\Control{k}{2}$ for the next iteration; otherwise our prediction does not
match the expectation and we decrease $\Control{k}{2}$. 

 Moreover, unlike the typical stochastic
convergence rate analysis, which bounds expected improvement  in either
$\Exp(\|\nabla f(x)\|) $ or $\Exp(f(x) -f_{\max} )$ after a
given number of iteration, our convergence rate analysis bounds the
total expected number of steps that the algorithm takes before either
$\|\nabla f(x)\|\leq \varepsilon$ or $f(x) -f_{\max}\leq \varepsilon$  is reached. Our results rely on a stochastic
process framework introduced and analyzed in \cite{TR_prob_model} to provide convergence rates for stochastic trust region method. 



\subsection{Random gradient and function estimates} \label{sec:rand_grad_function} \paragraph{Overview.} At each iteration,
we compute a stochastic gradient and stochastic function
values. With probability $\pg$, the random direction $G_k$ is close to
the true gradient. We measure closeness or accuracy of the random
direction using the current step length, which is a known quantity. 
This procedure  \textit{naturally adapts} the required accuracy as the algorithm progresses. 
As the
steps get shorter (i.e. either the gradient gets smaller or the step-size parameter does), we require the
accuracy to increase, but the probability $\pg$ of encountering a good gradient $G_k$
at any iteration is the same. 

A similar procedure applies to function
estimates, $F_k^0$ and $F_k^s$. The accuracy of the function estimates
to the true function values at the points $x_k$ and $x_{k+1}$
are tied to the size of the step, $\mathcal{A}_k \norm{G_k}$. At each
iteration, there is a probability $\pf$ of obtaining good function
estimates. By choosing the probabilities of good gradient and
estimates, we show Algorithm~\ref{alg:line_search_method}
converges. To formalize this procedure, we introduce the following.

\paragraph{Notation and definitions.} Algorithm~\ref{alg:line_search_method} generates a random process $\{G_k, X_k, \Stepsize{k}{}, \Control{k}{}, S_k, F_k^0, F_k^s\}$, in what follows we will denote all random quantities by capital letters and their realization by small letters. 
Hence random gradient estimate is denoted by $G_k$ and its realizations - by $g_k = G_k(\omega)$. Similarly,
let the random quantities $x_k = X_k(\omega)$ (iterates), $\stepsize{k}{} =
\Stepsize{k}{}(\omega)$ (stepsize), control size $\Control{k}{}(\omega) =
\control{k}{}$, and $s_k = S_k(\omega)$ (step)
denote their respective realizations. Similarly, we let $\{F_k^0,
F_k^s\}$ denote estimates of $f(X_k)$ and $f(X_k + S_k)$, with their
realizations denoted by $f_k^0 = F_k^0(\omega)$ and $f_k^s =
F_k^s(\omega)$. Our goal
is to show that under some assumptions on $G_k$ and $\{F_k^0, F_k^s\}$
the resulting stochastic process convergences with probability one and
at an appropriate rate. In particular, we assume that the estimates $G_k$ and  $F_k^0$ and $F_k^s$ are
sufficiently accurate with sufficiently high probability, conditioned
on the past.

To formalize the conditioning on the past, let $\mathcal{F}_{k-1}^{G
  \cdot F}$ denote the $\sigma$-algebra generated by the random
variables $G_0, G_1, \hdots, G_{k-1}$ and $F_0^0, F_0^s, F_1^0, F_1^s,
\hdots, F_{k-1}^0, F_{k-1}^s$ and let $\mathcal{F}_{k-1/2}^{G \cdot
  F}$ denote the $\sigma$-algebra generated by the random variables
$G_0, G_1, \hdots, G_k$ and $F_0^0, F_0^s, F_1^0, F_1^s,
\hdots, F_{k-1}^0, F_{k-1}^s$. For completeness, we set
$\mathcal{F}_{-1}^{G \cdot F} = \sigma(x_0)$. As a result, we have that
$\mathcal{F}_k^{G \cdot F}$ for $k \ge -1$ is a filtration. By construction of the random variables $X_k$ and
$\Stepsize{k}{}$ in Algorithm~\ref{alg:line_search_method}, we see $\Exp [ X_k |
\mathcal{F}_{k-1}^{G \cdot F} ] = X_k$ and $\Exp [\Stepsize{k}{} |
\mathcal{F}_{k-1}^{G \cdot F}] = \Stepsize{k}{}$ for all $k \ge 0$. 

We measure accuracy of the gradient estimates $G_k$ and function estimates $F_k^0$ and
$F_k^s$ using the following definitions.

\begin{definition}\label{def:random_good_dir}
\rm{We say that a sequence of random  directions $\{ G_k\}$ is $(\pg)$-probabilistically
 $\constG$-sufficiently accurate for Algorithm \ref{alg:line_search_method} for the corresponding sequence $\{\Stepsize{k}{}, X_k\}$, if there exists a constant $\constG>0$,  such that the events
\[\goodModel = \{\|G_k-\nabla f(X_k)\|\leq \constG \Stepsize{k}{} \|G_k\|\} \]
satisfy the conditions\footnote{Given a measurable set $A$, we use
  $1_{A}$ as the indicator function for the set $A$; $1_{A} =1$ if $\omega \in A$ and
  $0$ otherwise.}
\[ \Prob(\goodModel | \pastone) = \Exp [1_{\goodModel} |
  \pastone] \ge p_g\]}
\end{definition}

%

In addition to sufficiently accurate gradients, we require estimates on
the function values $f(x_k)$ and $f(x_k + s_k)$ to also be sufficiently
accurate. 

\begin{definition}\rm{ A sequence of random estimates $\{F_k^0,
  F_k^s\}$ is said to be \textit{$\pf$-probabilistically
  $\constF$-accurate} with respect to the corresponding sequence
  $\{X_k, \Stepsize{k}{}, S_k\}$ if the events
\[\goodFunction = \{ |F_k^0-f(x_k)| \le \constF \Stepsize{k}{2} \norm{G_k}^2 \quad \text{and}
  \quad |F_k^s - f(x_k + s_k)| \le \constF \Stepsize{k}{2} \norm{G_k}^2\}.\]
satisfy the condition
\[\Prob( \goodFunction |\pasthalf) = \Exp [1_{\goodFunction} |
  \pasthalf ] \ge \pf.\]
  }
\end{definition}
We note here that the filtration $\pasthalf$ includes $
\Stepsize{k}{}$ and $G_k$; hence the accuracy of the estimates is
measured with respect to fixed quantities. Next, we state the key assumption on the nature of the stochastic information in Algorithm~\ref{alg:line_search_method}.

\begin{assumption} \label{assumpt:key} \rm{The following hold for the
    quantities in the algorithm:
\begin{enumerate}[(i)]
\item The sequence of random gradients $G_k$ generated by
  Algorithm~\ref{alg:line_search_method} is $\pg$-probabilistically 
 $\constG$-sufficiently accurate
  for some sufficiently large $\pg \in (0,1]$. 
\item The sequence of estimates $\{F_k^0, F_k^s\}$ generated by
  Algorithm~\ref{alg:line_search_method} is $\pf$-probabilistically
  $\constF$-accurate estimates for some $\constF \le
  \frac{\theta}{4 \alpha_{\max}}$ and sufficiently large $\pf \in
  (0,1]$.
\item The sequence of estimates $\{F_k^0, F_k^s\}$ generated by
  Algorithm~\ref{alg:line_search_method} satisfies a $\varConstF$-variance condition
  for all $k \ge 0$\footnote{We implicitly assume $|F_k^s-f(X_k+S_k)|^2$
  and $|F_k^0-f(X_k)|^2$ are integrable for all $k$; thus it is
  straightforward to deduce $|F_k^s-f(X_k+S_k)|$ and
  $|F_k^0-f(X_k)|$ are integrable for all $k$.}, 
\begin{equation} \label{eq: variance} \begin{aligned} 
&\Exp [ |F_k^s-f(X_k + S_k)|^2 | \pasthalf] \le
  \max \{\varConstF^2 \Stepsize{k}{2} \norm{\nabla f(X_k)}^4, \theta^2
  \Control{k}{4} \}\\
 \text{and} \qquad &\Exp [|F_k^0-f(X_k)|^2 | \pasthalf] \le \max
                                                           \{\varConstF^2
                                                           \Stepsize{k}{2}
                                                           \norm{\nabla
                                                           f(X_k)}^4,
                                                           \theta^2
                                                           \Control{k}{4} \}.
\end{aligned}
\end{equation}
\end{enumerate}
}
\end{assumption}
A simple calculation shows that under Assumption \ref{assumpt:key} the following hold
\[ \Exp [1_{\goodModel \cap \goodFunction}| \pastone]  \ge \pg \pf,\quad
\Exp [1_{\goodModel^c \cap \goodFunction} | \pastone]   \le
                                                               1-\pg,
\quad \text{and} \quad \Exp [1_{\goodFunction^c} | \pastone]  \le 1-\pf.
\]
\begin{remark} \label{rmk:constants} \rm{We are interested in deriving
    convergence results for the case when $\constG$ may be large. For
    the rest of the exposition, without loss of generality $\constG
    \ge 2$. It clear if $\constG$ happens to be smaller, somewhat
    better bounds that the ones we derive here will result since the
    gradients give tighter approximations of the true
    gradient.  We are interested in deriving bound for the case
    when $\constG$ is large. Equation~\eqref{eq: variance} includes the maximum of
    two terms - one of the terms $\norm{\nabla f(X_k)}$ is
    unknown. When one posesses external knowledge of $\norm{\nabla f(X_k)}$,
    one could use this value. This is particularly useful  when $\norm{\nabla
      f(X_k)}$ is big since it allows large variance in the function
    estimates, for example assumption that $\norm{\nabla
      f(X_k)}\geq \varepsilon$ implies that this variance does not have to be driven to zero, before
       the algorithm reaches a desired accuracy. Yet, for convergence and since a useful lower bound
       on $\norm{\nabla f(X_k)}$ may be unknown, we
    include the parameter $\Control{k}{}$ as a way to adaptively control the variance. As such $\varConstF$ should be small, in fact, can be set
    equal to $0$. The analysis can be performed for any other values of the above constants - the choices here are for simplicity and convenience. }\end{remark}

This assumption on the accuracy of the gradient and function estimates
is key in our convergence rate analysis. We derive specific bounds on
$\pg$ and $\pf$ under which these rates would hold. We note here that
if $\pf=1$ then Assumption \ref{assumpt:key}(iii) is not needed and
condition $\pg>1/2$ is sufficient for the convergence results. This
case can be considered as an extension of results in
\cite{CartisScheinberg}. Before concluding this section, we state a result showing the
relationship between the variance assumption on the function values
and the probability of inaccurate estimates.

\begin{lem} \label{lem:uniform_integrable} Let Assumption~
  \ref{assumpt:key} hold. Suppose $\{X_k, G_k, F_k^0, F_k^s,
  \Stepsize{k}{}, \Control{k}{}\}$ is a random process generated by
  Algorithm~\ref{alg:line_search_method} and $\{F_k^0, F_k^s\}$ are
  $\pf$-probabilistically accurate estimates. Then for every $k \ge
  0$ we have
\begin{align*}
&\Exp [1_{\goodFunction^c} |F_k^s-f(X_k + S_k)|~| \pasthalf ] \le \left ( 1-\pf \right )^{1/2} \max \{ \varConstF
  \Stepsize{k}{}\norm{\nabla f(X_k)}^2, \theta \Control{k}{2} \}\\
\text{and} \qquad &\Exp [1_{\goodFunction^c} |F_k^0-f(X_k)|~| \pasthalf] \le \left ( 1-\pf \right )^{1/2} \max \{ \varConstF
  \Stepsize{k}{}\norm{\nabla f(X_k)}^2, \theta \Control{k}{2} \}.
\end{align*}
\end{lem}

\begin{proof} We show the result for $F_k^0-f(X_k)$, but the proof
  for $F_k^s-f(X_k + S_k)$ is the same. Using Holder's inequality for conditional expectations,
  we deduce
\begin{align*}
\Exp \left [ \tfrac{1_{\goodFunction^c} | F_k^0-f(X_k)|}{\max \{
  \varConstF \Stepsize{k}{} \norm{\nabla f(X_k)}^2, \theta
  \Control{k}{2} \} } \big | \pasthalf \right ]
  \le \left ( \Exp[1_{\goodFunction^c} | \pasthalf ] \right )^{1/2} \left ( \Exp \left [\tfrac{ | F_k^0-f(X_k)|^2}{\max \{
  \varConstF^2 \Stepsize{k}{2} \norm{\nabla f(X_k)}^4, \theta^2
  \Control{k}{4} \} } \big | \pasthalf \right ]
  \right )^{1/2}.
\end{align*}
The result follows after noting by \eqref{eq: variance}
\[\left ( \Exp \left [\tfrac{ | F_k^0-f(X_k)|^2}{\max \{
  \varConstF^2 \Stepsize{k}{2} \norm{\nabla f(X_k)}^4, \theta^2
  \Control{k}{4} \} } \big | \pasthalf \right ]
  \right )^{1/2} \le 1.\]
\end{proof}


\subsection{Computing $G_k$, $F_k^0$, and $F_k^s$ to satisfy Assumption \ref{assumpt:key}.}
Assuming that the variance of random function and gradient realizations
is bounded as 
$$
\Exp(\|\nabla \tilde f(x,\xi_i)-\nabla f(x)\|^2)\leq V_g\ \text{and\ } \Exp( |\tilde f(x,\xi_i)-f(x)|^2)\leq V_f,
$$
 Assumption
\ref{assumpt:key} can be made to  hold if $G_k$, $F_k^0$ and $F_k^s$ are computed using a sufficient number of
samples. 
In particular,  let $S_k$ be a sample of realizations $\nabla f(x,\xi_i)$, $i\in S_k$ and 
 $G_k:= \tfrac{1}{|S_k|} \sum_{i\in S_k}\nabla \tilde f(X_k,\xi_i)$.
By using results e.g. in \cite{Tripuraneni2017,Tropp2015} we can show that if 
\begin{equation}\label{eq:Sk}
|S_k|\geq \tilde O(\tfrac{V_g}{\constG^2\Stepsize{k}{2} \|G_k\|^2})
\end{equation}
 (where $\tilde O$ hides the log factor of ${1}/(1-\pg)$), then 
 Assumption \ref{assumpt:key}(i) is satisfied. 
 While $G_k$ is not known when $|S_k|$ is chosen, one can design a simple loop by guessing the value of 
$\|G_k\|$  and increasing the number of samples until \eqref{eq:Sk} is satisfied, this procedure is discussed in \cite{CartisScheinberg}. 
Similarly to satisfy Assumption~\ref{assumpt:key}(ii), it is sufficient to compute
$F^0_k= \tfrac{1}{|S_k^0|} \sum_{i\in S_k^0} \tilde f(X_k,\xi_i)$ with 
\[
|S^0_k|\geq \tilde O(\tfrac{V_f}{\varConstF^2 \Stepsize{k}{2}\|G_k\|^4})
\]
 (where $\tilde O$ hides the log factor of ${1}/(1-\pf)$)
and to obtain $F^s_k$ analogously. 
Finally, it is easy to see that Assumption  \ref{assumpt:key}(iii) is simply satisfied if 
$|S^0_k|\geq \tfrac{V_f}{\theta^2\Control{k}{4}}$ by standard properties of variance. 

We observe that: 
\vskip-0.1in
\begin{itemize}
\item unlike \cite{RByrd_etal_2012, Schmidt}, the number of samples for gradient and function estimation does not increase at any pre-defined rate, but is closely related to the progress of the algorithm. In particular if $\Stepsize{k}{} \|G_k\|$ and $\Control{k}{}$ increase then the sample sets sizes can decrease.
\item 
Also, unlike  \cite{Tripuraneni2017}  where the number of samples is simply chosen large enough a priori for all $k$ so that the right hand side in Assumption~\ref{assumpt:key}(i) is bounded by a predefined accuracy $O(\varepsilon)$, our algorithm can be applied without knowledge of $\varepsilon$. 
\item 
  Finally, unlike  \cite{Bollapragada2017} where   theoretical results require  that $|S_k|$ depends on 
$\|\nabla f(X_k)\|$, which is  unknown, our bounds on the sample set sizes all use knowable quantities, such as 
bound on the variance and quantities computed by the algorithm. 
\end{itemize}

We also point out $\constG$ can be arbitrarily big and $\pg$ depends only on the backtracking factor 
$\gamma$ and is not close to $1$; hence the number of samples to
satisfy  Assumption \ref{assumpt:key}(i) is moderate. On the other
hand, $\pf$ will have to depend on $\constG$; hence a looser control of 
the gradient estimates results in tighter control, \textit{i.e.} larger sample sets, for function estimates. 

Our last comment is that $G_k$ does not have to be an unbiased estimate of $\nabla f(X_k)$ and does not need to be computed via gradient samples. Instead it can be computed via stochastic finite differences, as is discussed for example in 
\cite{ChenMenickellyScheinberg2014}.

\section{Renewal-Reward Process} In this section, we
define a general random process  introduced in \cite{TR_prob_model} and its stopping time $T$
which serve as a general framework for analyzing behavior of stochastic trust region method in 
\cite{TR_prob_model} and stochastic line search in this paper.   We state the relevant definitions,
assumptions, and theorems and refer the reader to the proofs in
\cite{TR_prob_model}.

\begin{definition} \rm{Given a discrete time stochastic process $\{X_k\}$, a random
    variable $T$ is a \textit{stopping time} for $\{X_k\}$ if the event $\{T = k\} \in \sigma(X_0,\hdots,
    X_k)$. 
}
\end{definition}

 Let $\{\Phi_k, \mathcal{A}_k\}$ be a random process such
that $\Phi_{k}\in\lbrack0,\infty)$ and $\mathcal{A}_{k}\in\lbrack0,\infty)$ for
$k\geq0$.  
  Let us also define a
 biased random walk process, $\{W_k\}_{k=1}^\infty$, defined on the same probability 
 space as $\{\Phi_k, \mathcal{A}_k\}$. We denote $\mathcal{F}_k$
 the $\sigma$-algebra generated by $\{\Phi_0, \mathcal{A}_0, W_0, \hdots,
 \Phi_k, \mathcal{A}_k, W_k\}$, where $W_0=1$. In addition, $W_k$ obeys the following dynamics
\begin{equation}\label{w_process}
{\bf{Pr}}(W_{k+1} = 1 | \mathcal{F}_{k}) = p \quad \text{and} \quad
  {\bf{Pr}}(W_{k+1} = -1 | \mathcal{F}_{k}) = (1-p) \quad 
    \end{equation}

We define $T_\varepsilon$ to be a family of
 stopping times parameterized by $\varepsilon$. 
In \cite{TR_prob_model} a bound on  $\Exp(T_{\varepsilon})$ is derived under the following assumption on 
$\{\Phi_k, \mathcal{A}_k\}$. 

\begin{assumption}\label{assump: random_variable} \rm{
The following hold for the process $\{\Phi_k, \mathcal{A}_k, W_k\}$.
\begin{enumerate}[(i)]
\item $\mathcal{A}_{0}$ is a constant.  There exists a constant $\lambda\in\left(  0,\infty\right)  $ and
$\alpha_{\max}=\mathcal{A}_{0}e^{  \lambda j_{\max}}  $ (for some
$j_{\max}\in\mathbb{Z}$) 
such that $\mathcal{A}_{k}\leq\alpha_{\max}$ for all $k\geq 0$.

\item There exists a constant $\bar{\mathcal{A}} = \mathcal{A}_0e^{  \lambda \bar j}$
for some $\bar j\in Z$ and $\bar j\leq 0$,  such
  that, the following holds for all $k\geq 0$, 
\[ 
1_{ \{T_{\epsilon}>k\} }\mathcal{A}_{k+1} \ge1_{\{ T_{\epsilon}>k\}} \min \left \{ \mathcal{A}_k e^{\lambda W_{k+1}},
  \bar{\mathcal{A}} \right \} \]
  where $W_{k+1}$ satisfies \eqref{w_process} with $p>\frac{1}{2}$.
  
\item There exists a nondecreasing function $h :[0,\infty
)\rightarrow(0,\infty)$ and a constant $\Theta>0$  such that
\[ 
1_{ \{T_{\epsilon}>k\}}\text{\textbf{E}} \left [ \Phi_{k+1} | \mathcal{F}_{k}
  \right ] \le 1_{ \{T_{\epsilon}>k\}}(\Phi_{k} - \Theta
  h(\mathcal{A}_{k})). \]
\end{enumerate}}
\end{assumption}

Assumption~\ref{assump: random_variable} (iii) states that conditioned on the
event $T_{\varepsilon} > k$ and the past, the random variable $\Phi_k$
decreases by $\Theta h(\mathcal{A}_{k})$ at each iteration. Whereas
Assumption~\ref{assump: random_variable} (ii) says that once
$\mathcal{A}_k$ falls below the fixed constant $\bar{\mathcal{A}}$,
the sequence has a tendency to increase. Assumptions~\ref{assump: random_variable} (i) and (ii) together
also  ensures that $ \bar{\mathcal{A}}$ belongs to the sequence of values taken by the sequence $\mathcal{A}_{k}$. 
As we will see this is a simple technical assumption that can be satisfied w.l.o.g.

\begin{remark} \rm{Computational complexity (in deterministic methods)
    measures the number of
    iterations until an event such as $\norm{\nabla f(x)}$ is small or
    $f(x_k)-f^*$ is small,
    or equivalently, the rate at which the gradient/function values decreases as a
    function of the iteration counter $k$. For randomized or
    stochastic methods, previous works tended to focus on the second definition,
    i.e. showing the expected size of the gradient or function values decreases like
    $1/k$. Instead, here we bound the expected number of
    iterations until the size of the gradient or function values are small, which is
    the same as bounding 
   the stopping times $T_{\varepsilon} = \inf \{ k\ge 0 \, : \,
    \norm{\nabla f(X_k)} < \varepsilon\}$ and $T_{\varepsilon} = \inf
    \{ k \ge 0 \, : \, f(X_k)-f^* \le \varepsilon \}$, for a fixed $\varepsilon > 0$. 
}
\end{remark}

\begin{remark} \rm{In the context of deterministic line search, when the
  stepsize $\alpha_k$ falls below the constant $1/L$, where $L$ is the
  Lipschitz constant of $\nabla f(x)$, the iterate $x_k + s_k$ always
  satisfies the sufficient decrease condition, namely $f(x_k + s_k)
  \le f(x_k) - \theta \alpha_k \norm{\nabla f(x_k)}^2$. Thus $\alpha_k$ never falls much below $1/L$. 
  To match the
  dynamics behind deterministic line search, we expect $\Phi_{k+1} - \Phi_k \approx f(X_{k+1}) - f(X_k)$ with
  $\Theta h(\mathcal{A}_k) \approx  \mathcal{A}_k \norm{\nabla
    f(X_k)}^2$ and the constant $\bar{\mathcal{A}} \approx
  1/L$. However,  in the stochastic setting  there is a positive
probability of $\mathcal{A}_k$ being arbitrarily small. 
Theorem \ref{thm:renewal_reward_stop_time}, below, 
is derived by observing that on average $\mathcal{A}_k \ge \bar{\mathcal{A}}$ occurs
frequently due to the upward drift in the random walk
process. Consequently, ${\bf{E}}[\Phi_{k+1}-\Phi_k]$ can be bounded by
a negative fixed value (dependent on $\varepsilon$)
frequently; thus we can derive a bound on
${\bf{E}}[T_{\varepsilon}]$. 
}
\end{remark}

The following theorem (Theorem  2.2 in \cite{TR_prob_model})  bounds   ${\bf{E}}[T_{\varepsilon}]$ in terms of
$h(\bar{\mathcal{A}})$ and $\Phi_0$. 
\begin{thm}\label{thm:renewal_reward_stop_time} Under Assumption~\ref{assump: random_variable}, 
\[ {\bf{E}}[T_{\varepsilon}] \le \frac{p}{2p-1} \cdot
    \frac{\Phi_0}{\Theta h(\bar{\mathcal{A}})} 
    + 1 .\]
\end{thm}

\section{Convergence of Stochastic Line Search}\label{sec:convergence} Our primary goal is to prove
convergence of Algorithm~\ref{alg:line_search_method} by showing
a \textit{lim-inf convergence} result, $\liminf_{k \to \infty} \norm{\nabla f(X_k)}
= 0$ a.s. We that typical convergence results for
stochastic algorithms prove either high probability results or that the
expected gradient at an averaged point converges. Our result is
slightly stronger than these results since we show a subsequence of
the $\norm{\nabla f(X_k)}$ converges a.s. With this convergence result, stopping times based on either
$\norm{\nabla f(x)} < \varepsilon$ and/or $f(x)-f_{\min} <
\varepsilon$ are finite almost surely. Our approach for the liminf
proof is twofold: 
(1) construct a function $\Phi$ ($\approx f$) whose expected progress decreases
  proportionally to $\norm{\nabla f(x)}^2$ and (2) the $\limsup$ of
  the step sizes is strictly larger than $0$ a.s. 

\subsection{Useful results} \label{sec:useful_results} Before delving
into the convergence statement and proof, we state some lemmas similar
to those derived in \cite{CartisScheinberg,TR_prob_models_good_fun,ChenMenickellyScheinberg2014}.
 
\begin{lem}[Accurate gradients $\Rightarrow$ lower bound on
  $\norm{g_k}$] \label{lem: lower_bound_g}
Suppose $g_k$ is $\constG$-sufficiently accurate. Then 
\[ \frac{\norm{\nabla f(x_k)}}{(\constG
    \stepsize{\max}{} + 1)} \le \norm{g_k}. \]
\end{lem}

\begin{proof} Because $g_k$ is $\constG$-sufficiently accurate together with the triangle
  inequality implies
\begin{align*}
\norm{\nabla f(x_k)} \le (\constG
  \stepsize{k}{} +1)\norm{g_k} \le (\constG
  \stepsize{\max}{} + 1) \norm{g_k}.
\end{align*}
\end{proof}



\begin{lem}[Accurate gradients and estimates $\Rightarrow$
  successful iteration] \label{lem: good_est_good_model_successful}
  Suppose $g_k$ is $\constG$-sufficiently accurate and $\{f_k^0, f_k^s\}$ are $\constF$-accurate estimates. If 
\[\stepsize{k}{} \le  \frac{1-\theta}{\constG + \frac{L}{2}
      + 2\constF}\]
then the trial step $x_k + s_k$ is successful. In particular, this means $f_k^s \le f_k^0 - \theta \stepsize{k}{} \norm{g_k}^2.$
\end{lem}

\begin{proof} The $L$-smoothness of $f$ and the $\constG$-sufficiently
  accurate gradient immediately yield 
\begin{align*}
f(x_k + s_k) &\le f(x_k) - \stepsize{k}{} (\nabla f(x_k)-g_k)^Tg_k - \stepsize{k}{}
  \norm{g_k}^2 + \tfrac{L \stepsize{k}{2}}{2} \norm{g_k}^2\\
&\le f(x_k) + \constG \stepsize{k}{2} \norm{g_k}^2 - \stepsize{k}{} \norm{g_k}^2
+ \tfrac{L \stepsize{k}{2}}{2} \norm{g_k}^2.
\end{align*}
Since the estimates are $\constF$-accurate, we obtain
\begin{align*}
f_k^s - \constF \stepsize{k}{2} \norm{g_k}^2 &\le f(x_k + s_k) - f_k^s + f_k^s\\
  &\le f(x_k) - f_k^0 + f_k^0 +\constG\stepsize{k}{2} \norm{g_k}^2 - \stepsize{k}{}
  \norm{g_k}^2 + \tfrac{L \stepsize{k}{2}}{2} \norm{g_k}^2 \\
&\le f_k^0 + \constF \stepsize{k}{2} \norm{g_k}^2 + \constG \stepsize{k}{2} \norm{g_k}^2 - \stepsize{k}{}
  \norm{g_k}^2 + \tfrac{L \stepsize{k}{2}}{2} \norm{g_k}^2.
\end{align*}
The result follows by noting $f_k^s \le f_k^0 - \stepsize{k}{} \norm{g_k}^2 \left ( 1- \stepsize{k}{} \left (
  \constG + \frac{L}{2} + 2\constF
  \right ) \right )$.
\end{proof}

\begin{lem}[Good estimates $\Rightarrow$ decrease in function] \label{lem: estimate_good_funct_decrease}  Suppose
  $\constF < \frac{\theta}{4 \stepsize{\max}{}}$ and $\{f_k^s,f_k^0\}$ are $\constF$-accurate estimates. If the trial step is
  successful, then the improvement in function value is
\begin{equation} \label{eq: good_est_success}
f(x_{k+1}) \le f(x_k) - \frac{\theta \stepsize{k}{}}{2} \norm{g_k}^2 .
\end{equation}
If, in addition, the step is reliable, then the improvement
in function value is 
\begin{equation} \label{eq: good_est_true}
f(x_{k+1}) \le f(x_k) - \frac{\theta \stepsize{k}{}}{4} \norm{g_k}^2 -
\frac{\theta}{4} \control{k}{2}.
\end{equation}
\end{lem}

\begin{proof} The iterate $x_k + s_k$ is successful and the estimates
  are $\constF$ accurate so we conclude
\begin{align*}
f(x_k + s_k) &\le f(x_k + s_k) - f_k^s +f_k^0 - f(x_k) + f(x_k) - \stepsize{k}{}
  \theta \norm{g_k}^2\\
&\le f(x_k) + 2 \constF \stepsize{k}{2}
                    \norm{g_k}^2 - \stepsize{k}{} \theta \norm{g_k}^2\\
&\le f(x_k) - \stepsize{k}{} \norm{g_k}^2 \left ( \theta-2 \constF \stepsize{\max}{}
  \right ),
\end{align*}
where the last inequality follows because $\stepsize{k}{} \le
\stepsize{\max}{}$. The condition $\constF < \frac{\theta}{4
  \stepsize{\max}{}}$ immediately implies \eqref{eq:
  good_est_success}. By noticing $\tfrac{\theta \stepsize{k}{}}{2}
\norm{g_k}^2 \ge \tfrac{\theta \stepsize{k}{}}{4} \norm{g_k}^2 +
\tfrac{\theta \control{k}{2}}{4}$ holds for reliable steps, we deduce
\eqref{eq: good_est_true}. 
\end{proof}

\begin{lem}\label{lem: bound_gradient} Suppose the iterate is successful. Then 
\[ \norm{\nabla f(x_{k+1})}^2 \le 2 (L^2 \stepsize{k}{2} \norm{g_k}^2 +
  \norm{\nabla f(x_k)}^2 ). \]
In particular, the inequality holds
\[\tfrac{1}{L^2} \left ( \stepsize{k+1}{} \norm{\nabla f(x_{k+1})}^2 - \stepsize{k}{} \norm{\nabla
    f(x_k)}^2\right ) \le 2 \gamma \stepsize{k}{}   (\stepsize{\max}{2} \norm{g_k}^2 +
  \tfrac{1}{L^2}\norm{\nabla f(x_k)}^2 ).\]
\end{lem}
\begin{proof} An immediate consequence of $L$-smoothness of $f$ is $\norm{\nabla f(x_{k+1})} \le L \alpha_k \norm{g_k} +
  \norm{\nabla f(x_k)}$. The result follows from squaring both sides and applying the
  bound, $(a+b)^2 \le 2(a^2 + b^2)$. To obtain the second inequality, we note that in the case $x_k +
  s_k$ is successful, $\stepsize{k+1}{} = \gamma \stepsize{k}{}$. 
\end{proof}

\begin{lem}[Accurate gradients and estimates $\Rightarrow$ decrease in
  function] \label{lem:decrease_function} Suppose $g_k$ is
  $\constG$-sufficiently accurate and $\{f_k^0, f_k^s\}$ are $\constF$-accurate estimates where
  $\constF \le \tfrac{\theta}{4 \stepsize{\max}{}}$. If the trial step
    is successful, then 
\begin{equation} \label{eq:decrease_function_succ} f(x_{k+1})-f(x_k) \le -\frac{\theta \stepsize{k}{}}{4}
  \norm{g_k}^2 - \frac{\theta \stepsize{k}{}}{4(\constG
    \stepsize{\max}{} +1)^2} \norm{\nabla f(x_k)}^2. 
\end{equation}
In addition, if the trial step is reliable, then
\begin{equation}\label{eq:decrease_function_succ_true} f(x_{k+1})-f(x_k) \le  -\frac{\theta \stepsize{k}{}}{8}
  \norm{g_k}^2- \frac{\theta }{8} \control{k}{2}  - \frac{\theta \stepsize{k}{} }{4(\constG
    \stepsize{\max}{} +1)^2} \norm{\nabla f(x_k)}^2.\end{equation}
\end{lem}

\begin{proof} Lemma~\ref{lem: lower_bound_g} implies 
\begin{equation}\label{eq:blah}
-\tfrac{\theta}{2} \stepsize{k}{} \norm{g_k}^2 \le -\tfrac{\theta}{4}
\stepsize{k}{} \norm{g_k}^2 - \tfrac{\theta}{4(\constG
  \stepsize{\max}{} + 1)^2} \stepsize{k}{} \norm{\nabla f(x_k)}^2.
\end{equation}
We combine this result with Lemma~\ref{lem:
  estimate_good_funct_decrease} to conclude the first result. For the
second result, since the step is reliable,
equation~\eqref{eq:blah} improves to 
\begin{align*}
-\tfrac{\theta}{2} \stepsize{k}{} \norm{g_k}^2 \le -\tfrac{\theta}{8}
\stepsize{k}{} \norm{g_k}^2 - \tfrac{\theta}{8} \control{k}{2} - \tfrac{\theta}{4(\constG
  \stepsize{\max}{} + 1)^2} \stepsize{k}{} \norm{\nabla f(x_k)}^2,
\end{align*}
and again the result follows from Lemma~\ref{lem:
  estimate_good_funct_decrease}.
\end{proof}

\subsection{Definition and analysis of $\{\Phi_k, \Stepsize{k}{}, W_k\}$ process for Algorithm~\ref{alg:line_search_method}}
We base our proof of convergence on properties of the random function
\begin{equation} \label{eq:phi} \Phi_k = \nu (f(X_k)-f_{\min}) + (1-\nu) \frac{1}{L^2}
  \Stepsize{k}{} \norm{\nabla f(X_k)}^2 + (1-\nu) \theta
  \Control{k}{2}.
\end{equation}
for some (deterministic) $\nu \in (0,1)$ and $f_{\min} \le f(x)$ for
all $x$. The goal is to show that $\{\Phi_k, \Stepsize{k}{}\}$ satisfies Assumption \ref{assump: random_variable}, in particular, that $\Phi_k$ is expected to decrease on each iteration.  Due to inaccuracy in function estimates and gradients, the
algorithm may take a step that increases the objective and thus $\Phi_k$. We will show that such increase if  bounded by a value proportional to $\norm{\nabla f(x)}^2$. On the other hand, as we will show, on successful iteration with accurate function estimates, the objective decreases proportionally  $\norm{\nabla f(x)}^2$, while on unsuccessful steps, equation
\eqref{eq:phi} is always negative because both $\Stepsize{k}{}$ and  $\Control{k}{}$ are decreased. The
function $\Phi$ is chosen to balance the potential increases and decreases in the
objective with changes inflicted by unsuccessful steps. 


\begin{thm} \label{thm: expected_decrease} Let
  Assumptions~\ref{assumpt: objective_function} and
  \ref{assumpt:key} hold. Suppose $\{X_k, G_k, F_k^0, F_k^s,
  \Stepsize{k}{}, \Control{k}{}\}$ is the random process
  generated by Algorithm~\ref{alg:line_search_method}. Then there exist
  probabilities $\pg, \pf > 1/2$ and a constant $\nu \in (0,1)$ such
  that the expected decrease in $\Phi_k$ is
\begin{equation} \label{eq: expect_non_convex_decrease} \Exp [\Phi_{k+1}-\Phi_k | \pastone] \le
  - \frac{\pg \pf (1-\nu) (1-\gamma^{-1})}{4} \left ( \frac{\Stepsize{k}{}}{L^2}
    \norm{\nabla f(X_k)}^2 + \theta \Control{k}{2} \right ). 
\end{equation}
In particular, the constant $\nu$ and probabilities $\pg, \pf > 1/2$ satisfy 
\begin{gather}
\frac{\nu}{1-\nu} \ge \max \left \{ \frac{32 \gamma
    \stepsize{\max}{2}}{\theta}, 16(\gamma-1), \frac{16 \gamma
    (\constG \stepsize{\max}{} + 1)^2}{\theta} \right \}, \label{eq:bound_nu}\\
\pg \ge \frac{2 \gamma}{1/2(1-\gamma^{-1}) +2 \gamma} \label{eq:bound_pg}\\
\text{and} \qquad \frac{\pg\pf}{\sqrt{1-\pf}} \ge \max \left \{
  \frac{ 8L^2\nu\varConstF + 16 \gamma(1-\nu)}{
    (1-\nu)(1-\gamma^{-1})}, \frac{8\nu}{(1-\nu)(1-\gamma^{-1})} \right \} \label{eq:bound_product}. 
\end{gather}
\end{thm}

\begin{table}[hbtp!]
\centering
   \begin{tabular}{|c||c || c || c |}
	\hline
& \multicolumn{3}{c|}{Upper bound on $\Exp [\Phi_{k+1}-\Phi_k]$}\\
\hline
& \begin{minipage}{0.2\textwidth} \begin{center}\vspace{0.1cm} Accurate gradients
    \\ Accurate functions\\ w/ prob. $\pg\pf$ \vspace{0.1cm} \end{center} \end{minipage}
& \begin{minipage}{0.2\textwidth} \begin{center} Bad gradients \\ Accurate
    functions \\ w/ prob. $(1-\pg)\pf$ \end{center} \end{minipage}
& \begin{minipage}{0.15\textwidth} \begin{center} Bad
    functions \\ w/ prob. $1-\pf$ \end{center} \end{minipage} \\
\hline
\begin{minipage}{0.1\textwidth} \begin{center}\vspace{0.1cm} Success \vspace{0.1cm}
\end{center} \end{minipage}
&  \begin{minipage}{0.25\textwidth} \begin{center}\vspace{0.1cm} $-\frac{\Stepsize{k}{}}{L^2} \norm{\nabla f(X_k)}^2 -
                        \Control{k}{2}$\\ {\textcolor{blue}{decrease}}
                        \vspace{0.1cm} \end{center} \end{minipage}& \begin{minipage}{0.25\textwidth} \begin{center}\vspace{0.1cm}
                        $\frac{\Stepsize{k}{}}{L^2} \norm{\nabla
                          f(X_k)}^2$\\ {\textcolor{red}{increase}}
                        \vspace{0.1cm} \end{center} \end{minipage} & \begin{minipage}{0.25\textwidth} \begin{center}\vspace{0.1cm} $\frac{\Stepsize{k}{}}{L^2} \norm{\nabla f(X_k)}^2 
                        + \Control{k}{2}$\\ {\textcolor{red}{increase}}
                        \vspace{0.1cm} \end{center} \end{minipage} \\
\hline
\begin{minipage}{0.1\textwidth} \begin{center} Unsuccess
\end{center} \end{minipage} & \begin{minipage}{0.25\textwidth} \begin{center}\vspace{0.1cm} $-\frac{\Stepsize{k}{}}{L^2} \norm{\nabla f(X_k)}^2 -
                        \Control{k}{2}$\\ {\textcolor{blue}{decrease}}
                        \vspace{0.1cm} \end{center} \end{minipage}& \begin{minipage}{0.25\textwidth} \begin{center}\vspace{0.1cm} $-\frac{\Stepsize{k}{}}{L^2} \norm{\nabla f(X_k)}^2 -
                        \Control{k}{2}$\\ {\textcolor{blue}{decrease}}
                        \vspace{0.1cm} \end{center} \end{minipage} &\begin{minipage}{0.25\textwidth} \begin{center}\vspace{0.1cm} $-\frac{\Stepsize{k}{}}{L^2} \norm{\nabla f(X_k)}^2 -
                        \Control{k}{2}$\\ {\textcolor{blue}{decrease}}
                        \vspace{0.1cm} \end{center} \end{minipage}\\
\hline
\begin{minipage}{0.12\textwidth} \begin{center}
    \vspace{0.1cm}\textbf{Overall worst case
      improv.} \vspace{0.1cm}
\end{center} \end{minipage} & {\textcolor{blue}{decrease}} &
                                                             {\textcolor{red}{increase}}
&{\textcolor{red}{increase}}\\
\hline
\end{tabular}
\caption{We summarize the proof of Theorem~\ref{thm:
    expected_decrease} by displaying the values of
  $\Phi_{k+1}-\Phi_k$. The proof considers cases: accurate grad./functions estimates, bad grad./accurate
  functions estimates, and
  bad function estimates. Each of these is further broken into whether
  the step was successful/unsuccessful. We
  summarize the expected upper bounds on $\Phi_{k+1}-\Phi_k$ up to constants.
  }
  \label{tbl:summary_convergence}
\end{table}

\begin{proof}[Proof of Theorem~\ref{thm: expected_decrease}] Our proof
  considers three separate cases: good gradients/good estimates, bad
  gradients/good estimates, and lastly bad estimates. Each of these cases
  will be broken down into whether a successful/unsuccessful step is
  reliable/unreliable. To simplify notation, we introduce three
  sets
\begin{gather*}
\text{Succ} := \{ X_k + S_k \text{ is successful, namely sufficient
  decrease occurs}\},\\
\text{R} := \{ X_k + S_k \text{ is reliable, \textit{i.e.} $\Stepsize{k}{}\norm{G_k}^2 \ge
    \Control{k}{2}$}\},\\
\text{and} \qquad \text{U} := \{ X_k + S_k \text{ is unreliable,
  \textit{i.e.} $\Stepsize{k}{}\norm{G_k}^2 <
    \Control{k}{2}$}\},
\end{gather*}
Using this notation we can write
\begin{align*}
\Exp [\Phi_{k+1}-\Phi_k | \pastone] = \Exp
  [( 1_{\goodModel \cap \goodFunction} + 1_{\goodModel^c \cap
  \goodFunction} + 1_{\goodFunction^c}) (\Phi_{k+1}-\Phi_k) |
  \pastone].
\end{align*}
For each case we will derive a bound on the expected decrease (increase) in $\Phi_k$. In particular, we will show that, under an appropriate choice of $\nu$, when the model and the estimates are good all three types of steps result in a decrease of $\Phi_k$  proportional to 
$\Stepsize{k}{} \norm{\nabla f(X_k)}^2$ and 
                        $\Control{k}{2}$, while when model is bad, but the estimates are good, $\Phi_k$ may increase by an amount proportional to  
                        $\Stepsize{k}{} \norm{\nabla f(X_k)}^2$. Finally, when both the model and estimates are both bad, the expected increase 
                        in $\Phi_k$ is bounded by an amount proportional to $\Stepsize{k}{} \norm{\nabla f(X_k)}^2$ and 
                        $\Control{k}{2}$. Thus, by choosing the right probability values for these events, we can ensure overall expected decrease. These  bounds are derived in the proof below and are  summarized  in Table~\ref{tbl:summary_convergence}.

\noindent \textbf{Case 1 (Accurate gradients and estimates,
$1_{\goodModel \cap \goodFunction} = 1$).} We will show that the
 $\Phi_k$ decreases no matter what type of step occurs and that the smallest decrease happens on the unsuccessful step. Thus this case dominates the other two and overall we 
conclude that
\begin{equation} \label{eq:case_1_exp} \Exp [1_{\goodModel \cap \goodFunction} (\Phi_{k+1}-\Phi_k) |
  \pastone ] \le -\pg \pf
  (1-\nu)(1-\gamma^{-1}) \left ( 
  \frac{\Stepsize{k}{}}{L^2} \norm{\nabla f(X_k)}^2 + \theta
  \Control{k}{2} \right ). \end{equation}
\begin{enumerate}[(i).]
\item \textit{Successful and reliable step ($\success \True = 1$).}
  The iterate is successful and both the gradient and function estimates
  are accurate so a decrease in the true objective occurs, specifically, \ref{eq:decrease_function_succ_true} from Lemma~\ref{lem:decrease_function}) applies:
\begin{equation} \label{eq:case_1_1_1}\begin{aligned}
1_{\goodModel \cap \goodFunction} &\success \True \nu
  (f(X_{k+1})-f(X_k))\\
&\le -\nu 1_{\goodModel \cap \goodFunction} \success \True \left (
  \frac{\theta \Stepsize{k}{}}{8} \norm{G_k}^2 + \frac{\theta}{8}
  \Control{k}{2} + \frac{\theta \Stepsize{k}{}}{4(\constG
  \stepsize{\max}{} +1)^2} \norm{\nabla f(X_k)}^2
  \right ).
\end{aligned} \end{equation}
As the iterate is successful, the term $\Stepsize{k}{} \norm{\nabla
  f(X_k)}^2$ may increase, but its change is bounded due to Lemma~\ref{lem: bound_gradient}:
\begin{equation} \label{eq:case_1_1_2} \begin{aligned}
1_{\goodModel \cap \goodFunction} &\success \True (1-\nu) \frac{1}{L^2}
\left ( \Stepsize{k+1}{} \norm{\nabla f(X_{k+1})}^2-\Stepsize{k}{}
  \norm{\nabla f(X_k)}^2 \right )\\
&\le 1_{\goodModel \cap \goodFunction} \success \True (1-\nu) 2\gamma
\Stepsize{k}{} \left ( \stepsize{\max}{2} \norm{G_k}^2 +
  \tfrac{1}{L^2} \norm{\nabla f(X_k)}^2 \right ).
\end{aligned}
\end{equation}
Lastly because we have a reliable step,
$\Control{k+1}{2}$ increases by $\gamma$. Consequently, we deduce that
\begin{equation} \label{eq:case_1_1_3} 
1_{\goodModel \cap \goodFunction} \success \True (1-\nu) \theta
(\Control{k+1}{2} - \Control{k}{2}) = 1_{\goodModel \cap
  \goodFunction} \success \True (1-\nu) \theta (\gamma-1) \Control{k}{2}.
\end{equation}
Without loss of generality, suppose $L^2 \ge 1$. We choose $\nu$ sufficiently large so that the term on the right hand side of \eqref{eq:case_1_1_1} dominates the right hand sides of   \eqref{eq:case_1_1_2}, and
\eqref{eq:case_1_1_3}, specifically, 
\begin{equation} \label{eq:case_1_1_nu} \begin{aligned}
-\frac{\nu \theta \Stepsize{k}{}}{8} \norm{G_k}^2 &+ (1-\nu)2 \gamma
\Stepsize{k}{} \stepsize{\max}{2} \norm{G_k}^2 \le -\frac{\nu \theta
  \Stepsize{k}{}}{16} \norm{G_k}^2,\\
-\frac{\nu \theta \Stepsize{k}{}}{4L^2(\constG \stepsize{\max}{}
  +1)^2} \norm{\nabla f(X_k)}^2 + &(1-\nu) \frac{2 \gamma \Stepsize{k}{}}{L^2} 
\norm{\nabla f(X_k)}^2 \le - \frac{\nu \theta
  \Stepsize{k}{}}{8L^2(\constG \stepsize{\max}{} +1)^2} \norm{\nabla f(X_k)}^2,\\
\text{and} \qquad -\frac{\nu \theta}{8} \Control{k}{2} &+ (1-\nu) (\gamma-1) \theta
\Control{k}{2} \le - \frac{\nu \theta}{16} \Control{k}{2}.  
\end{aligned}
\end{equation}
We combine Equations~\eqref{eq:case_1_1_1}, \eqref{eq:case_1_1_2}, and
\eqref{eq:case_1_1_3} to conclude
\begin{equation} \label{eq:case_1_1_all} \begin{aligned}
1_{\goodModel \cap \goodFunction} \success &\True (\Phi_{k+1}-\Phi_k)\\
&\le -1_{\goodModel \cap \goodFunction} \success \True\left (\frac{\nu \theta
  \Stepsize{k}{}}{8L^2(\constG \stepsize{\max}{} +1)^2} \norm{\nabla
  f(X_k)}^2 + \frac{\nu \theta}{16} \Control{k}{2} \right ).
\end{aligned}
\end{equation}
\item \textit{Successful and unreliable step ($\success \notTrue = 1$).}
  Because the iterate is successful and our gradient/estimates are accurate,
  we again apply 
 Lemma~\ref{lem:decrease_function} to bound $f(X_{k+1}) - f(X_k)$ but this time using 
  \eqref{eq:decrease_function_succ} which holds for unreliable steps. The possible increase from
  the $(1-\nu) \tfrac{1}{L^2} \norm{\nabla f(X_k)}^2$ term is the same
  as \eqref{eq:case_1_1_2} where we replace $\True$ with $\notTrue$
  since Lemma~\ref{lem: bound_gradient} still applies. Lastly with an
  unreliable step, the change in $\Control{k}{2}$ is 
\begin{equation} \label{eq:case_1_2_3} 
1_{\goodModel \cap \goodFunction} \success \notTrue
(1-\nu) \theta (\Control{k+1}{2} -\Control{k}{2}) \le -1_{\goodModel \cap
  \goodFunction} \success \notTrue (1-\nu)(1-\gamma^{-1}) \theta \Control{k}{2}.
\end{equation}
Therefore by choosing $\nu$ such that \eqref{eq:case_1_1_nu} holds, we have
that
\begin{equation}  \label{eq:case_1_2_all} \begin{aligned}
1_{\goodModel \cap \goodFunction} &\success \notTrue
(\Phi_{k+1}-\Phi_k)\\
&\le -1_{\goodModel \cap \goodFunction} \success
\notTrue \left ( \tfrac{\nu \theta \Stepsize{k}{}}{8L^2(\constG
  \stepsize{\max}{} +1)^2} \norm{\nabla f(X_k)}^2 +
(1-\nu)(1-\gamma^{-1}) \theta \Control{k}{2} \right ).
\end{aligned}
\end{equation}
\item \textit{Unsuccessful iterate ($\notsuccess =1$).} Because the
  iterate is unsuccessful, the change in the function values is $0$ and the constants $\Stepsize{k}{}$ and $\Control{k}{2}$ 
  decrease. Consequently, we deduce that
\begin{equation} \label{eq:case_1_3_all} 
1_{\goodModel \cap \goodFunction} \notsuccess (\Phi_{k+1}-\Phi_k) \le
-1_{\goodModel \cap \goodFunction} \notsuccess (1-\nu)(1-\gamma^{-1})
\left ( \tfrac{\Stepsize{k}{}}{L^2} \norm{\nabla f(X_k)}^2 + \theta \Control{k}{2} \right ).
\end{equation}
\end{enumerate}
We chose $\nu$ sufficiently large to ensure that the third case (iii), unsuccessful iterate
\eqref{eq:case_1_3_all},  provides the worst case decrease when compared to 
\eqref{eq:case_1_1_all} and \eqref{eq:case_1_2_all}.  Specifically $\nu$  is chosen so that
\begin{equation} \label{eq:nu} \begin{gathered}
\frac{-\nu \theta \Stepsize{k}{}}{8L^2 (\constG \stepsize{\max}{}
  +1)^2} \norm{\nabla f(X_k)}^2  \le -(1-\nu)(1-\gamma^{-1})
  \frac{\Stepsize{k}{}}{L^2} \norm{\nabla f(X_k)}^2\\
\text{and} \qquad \frac{-\nu \theta}{16} \Control{k}{2} \le
  -(1-\nu)(1-\gamma^{-1}) \theta \Control{k}{2}. 
\end{gathered}
\end{equation}

As such, we bounded the change in $\Phi_k$ in the case of accurate
gradients and estimates by
\begin{align} \label{eq:integrable_1}
1_{\goodModel \cap \goodFunction} (\Phi_{k+1}-\Phi_k) \le
  -1_{\goodModel \cap \goodFunction} (1-\nu)(1-\gamma^{-1}) \left (
  \frac{\Stepsize{k}{}}{L^2} \norm{\nabla f(X_k)}^2 + \theta \Control{k}{2}
  \right ).
\end{align}
We take conditional expectations with respect to $\pastone$ and using
Assumption~\ref{assumpt:key}, equation \eqref{eq:case_1_exp} holds. 

\textbf{Case 2 (Bad gradients and accurate estimates, $1_{\goodModel^c \cap
    \goodFunction} = 1$)}
Unlike the previous case, $\Phi_k$ ma increase,  since the step along an inaccurate probabilistic gradients
may not provide enough decrease 
 to cancel the
increase from the $\norm{\nabla f(X_k)}^2$. Precisely, the successful and
unreliable case dominates the worst case increase in $\Phi_k$:
\begin{equation} \label{eq:case_2_exp}
\Exp [ 1_{\goodModel^c \cap \goodFunction} (\Phi_{k+1} - \Phi_k) |
\pastone ] \le (1-\pg) (1-\nu) 
\frac{2\gamma \Stepsize{k}{}}{L^2} \norm{\nabla f(X_k)}^2.
\end{equation}
As before, we consider three separate cases. 
\begin{enumerate}[(i)]
\item \textit{Successful and reliable step $(\success
    \True =1)$}. A successful, reliable step with accurate
  function estimates but bad gradients has functional improvement
  (Lemma~\ref{lem: estimate_good_funct_decrease}, equation \eqref{eq: good_est_true}):
\[ 1_{\goodModel^c \cap \goodFunction} \success \True
  \nu(f(X_{k+1})-f(X_k)) \le  -1_{\goodModel^c \cap \goodFunction}
  \success \True \nu \left ( \frac{\Stepsize{k}{} \theta \norm{G_k}^2}{4}
    + \frac{\theta}{4} \Control{k}{2} \right ).\]
In contrast to \eqref{eq:case_1_1_1}, we
lose the $\norm{\nabla f(X_k)}^2$ term. A reliable, successful step
increases both constants $\Stepsize{k+1}{}$ and
$\Control{k+1}{2}$, leading to 
\eqref{eq:case_1_1_2} and \eqref{eq:case_1_1_3} with $1_{\goodModel \cap \goodFunction}$ replaced
 by $1_{\goodModel^c \cap \goodFunction}$.  Hence by choosing $\nu$ to
satisfy \eqref{eq:case_1_1_nu}, the dominant term in $\Phi_k$ is
\begin{equation} \label{eq:case_2_1_1}
\begin{aligned}
1_{\goodModel^c \cap \goodFunction} &\success \True (\Phi_{k+1}-\Phi_k)\\
&\le 1_{\goodModel^c \cap \goodFunction} \success \True \left (
  -\frac{\nu \theta \Stepsize{k}{}}{16} \norm{G_k}^2 - \frac{\nu
    \theta}{16} \Control{k}{2} + \frac{2\gamma(1-\nu)}{L^2}
  \Stepsize{k}{} \norm{\nabla f(X_k)}^2 \right ).
\end{aligned}
\end{equation}
\item \textit{Successful and unreliable step $(\success
    \notTrue = 1)$}.  Lemma~\ref{lem:
    estimate_good_funct_decrease} holds, but this time equation \eqref{eq:
    good_est_success} for unreliable steps applies. Moreover,  \eqref{eq:case_1_1_2}
  and \eqref{eq:case_1_2_3} that bound the change in the last two terms of $\Phi_k$ also apply. Again by choosing $\nu$ to satisfy
  \eqref{eq:case_1_1_nu}, we deduce 
\begin{equation} \label{eq:case_2_1_2}
\begin{aligned}
&1_{\goodModel^c \cap \goodFunction} \success \notTrue
(\Phi_{k+1}-\Phi_k)\\
&\le 1_{\goodModel^c \cap \goodFunction} \success \notTrue \left (
  -\tfrac{\nu \theta \Stepsize{k}{}}{16} \norm{G_k}^2 -
  (1-\nu)(1-\gamma^{-1}) \theta \Control{k}{2} + 
  \tfrac{2\gamma(1-\nu)}{L^2} \Stepsize{k}{} \norm{\nabla f(X_k)}^2
\right ).
\end{aligned}
\end{equation}
\item \textit{Unsuccessful $(\notsuccess = 1)$.} As in the previous
  case, equation \eqref{eq:case_1_3_all} holds. 
\end{enumerate}
The right hand sides of  \eqref{eq:case_2_1_1} and \eqref{eq:case_2_1_2}
 and
\eqref{eq:case_1_3_all} are trivially upper bounded by the positive term $\Stepsize{k}{}\norm{\nabla f(X_k)}^2$.  Hence, we conclude that
\begin{equation}\label{eq:integrable_2} 1_{\goodModel^c \cap \goodFunction} (\Phi_{k+1}-\Phi_k) \le 
  1_{\goodModel^c \cap \goodFunction} \frac{2\gamma(1-\nu)}{L^2}
  \Stepsize{k}{} \norm{\nabla f(X_k)}^2. 
\end{equation}
Inequality \eqref{eq:case_2_exp} follows by taking expectations with
respect to $\pastone$ and noting that $\Exp [1_{\goodModel^c \cap
  \goodFunction} | \mathcal{F}_{k-1}^{M\cdot F}] \le 1-\pg$ as in Assumption~\ref{assumpt:key}.

\textbf{Case 3 (Bad estimates, $1_{\goodFunction^c}=1$).} Inaccurate
estimates can cause the algorithm to accept a step which can lead
to an increase in $f$, $\Stepsize{}{}$, and $\Control{}{}$ and hence in  $\Phi_k$.  We control this
increase in $\Phi_k$ by bounding the variance in the function
estimates, as in \eqref{eq: variance}, which is the key reason for Assumption~\ref{assumpt:key}(iii). By
adjusting the probability of outcome (Case (3)) to be sufficiently small, we can ensure that in expectation
$\Phi_k$ is sufficiently reduced. Precisely, we will show 
\begin{equation} \label{eq:case_3_exp} \begin{aligned}
\Exp [1_{\goodFunction^c} (\Phi_{k+1}-\Phi_k) | &\pastone ] \le 2\nu(\sqrt{1-\pf}) \max \{ \varConstF \Stepsize{k}{}
\norm{\nabla f(X_k)}^2, \theta \Control{k}{2} \}\\
&+ (1-\pf) \frac{(1-\nu) 2\gamma}{L^2} \Stepsize{k}{} \norm{\nabla f(X_k)}^2.
\end{aligned}
\end{equation}
A successful step leads to the following bound
\begin{equation}\label{eq: case_3_success}
\begin{aligned}
1_{\goodFunction^c} \success \nu &\left ( f(X_{k+1})-f(X_k) \right ) \le
1_{\goodFunction^c} \success \nu \left ( (F_k^s - F_k^0) +
  |f(X_{k+1})-F_k^s| + |F_k^0-f(X_k)| \right )\\
&\le 1_{\goodFunction^c} \success \nu \left ( -\theta \Stepsize{k}{} \norm{G_k}^2 +
  |f(X_{k+1})-F_k^s| + |F_k^0-f(X_k)| \right ),
\end{aligned} 
\end{equation}
where the last inequality is due to the sufficient decrease condition. 
As before, we consider three separate cases. 
\begin{enumerate}[(i).]
\item \textit{Successful and reliable step $(\success \True =1)$.} With a
  reliable step we have $-\Stepsize{k}{} \norm{G_k}^2 \le
  -\Control{k}{2}$, thus \eqref{eq: case_3_success} implies 
\begin{align*}
1_{\goodFunction^c} &\success \True
  \nu(f(X_{k+1})-f(X_k))\\ 
&\le 1_{\goodFunction^c}
                                           \success \True \nu \left (
                                           -\tfrac{1}{2} \theta
                                           \Stepsize{k}{} \norm{G_k}^2
                                           - \tfrac{\theta}{2}
  \Control{k}{2} +|f(X_{k+1})-F_k^s|
  + |F_k^0-f(X_k)|
                                           \right ). 
\end{align*}
 We note that $\Phi_{k+1}-\Phi_k$ is upper bounded by the right hand side of the above inequality and the right hand sides of 
  \eqref{eq:case_1_1_2} and \eqref{eq:case_1_1_3}. 
As before, by choosing $\nu$ as in \eqref{eq:case_1_1_nu} we ensure $\tfrac{-\nu
  \theta}{2} \Stepsize{k}{} \norm{G_k}^2 + (1-\nu)2 \gamma \Stepsize{k}{}
\stepsize{\max}{2} \norm{G_k}^2 \le 0$ and $\tfrac{-\nu \theta}{2}
\Control{k}{2} + (1-\nu) (\gamma-1) \theta \Control{k}{2} \le 0$. 
It follows that
\begin{equation} \label{eq:case_3_1}
\begin{aligned}
1_{\goodFunction^c} \success \True &(\Phi_{k+1}-\Phi_k)\\ &\le
  1_{\goodFunction^c} \left (\nu |f(X_{k+1})-F_k^s| +
 \nu |F_k^0-f(X_k)| + (1-\nu)\tfrac{2\gamma}{L^2} \Stepsize{k}{}
  \norm{\nabla f(X_k)}^2 \right ) .
\end{aligned}
\end{equation}
\item \textit{Successful and unreliable step $(\success \notTrue =
    1)$.} Since on  unreliable steps, 
  $\Control{k+1}{2}$ is decreased, then the increase in $\Phi_k$ is always smaller than the worst-case increase we just derived for the successful and reliable 
  step. Thus \eqref{eq:case_3_1} holds with $\True$ replaced by $\notTrue$. 

\item \textit{Unsuccessful $(\notsuccess = 1)$} As we decrease both
  $\Control{}{}$ and $\Stepsize{}{}$, and $X_{k+1}= X_k$, we conclude
  that \eqref{eq:case_1_3_all} hold. 
\end{enumerate}

The equation \eqref{eq:case_3_1} dominates \eqref{eq:case_1_3_all};
thus in all three cases \eqref{eq:case_3_1} holds. We
take expectations of \eqref{eq:case_3_1} and apply Lemma~\ref{lem:uniform_integrable} to
conclude that
\begin{equation} \label{eq:case_3_1_exp}
\begin{aligned} \Exp[1_{\goodFunction^c} (\Phi_{k+1}-\Phi_k) |
  \pastone ] &\le 2 \nu (1-\pf)^{1/2} \max\{
  \varConstF \Stepsize{k}{} \norm{\nabla f(X_k)}^2, \theta
  \Control{k}{2} \}\\
& (1-\pf) (1-\nu) \tfrac{2\gamma}{L^2} \Stepsize{k}{} \norm{\nabla f(X_k)}^2.
\end{aligned}
\end{equation}
Now we combine the expectations \eqref{eq:case_1_exp}, \eqref{eq:case_2_exp}, and \eqref{eq:case_3_exp} to obtain
\begin{align*}
\Exp [\Phi_{k+1}-\Phi_k &| \pastone] = \Exp[ (1_{\goodModel \cap \goodFunction} + 1_{\goodModel^c \cap \goodFunction} + 1_{\goodFunction^c})(\Phi_{k+1}-\Phi_k)| \pastone]\\
&\le -\pg\pf (1-\nu) (1-\gamma^{-1}) \left ( \tfrac{\Stepsize{k}{}}{L^2} \norm{\nabla f(X_k)}^2 + \theta \Control{k}{2} \right ) + \pf (1-\pg) \tfrac{2\gamma (1-\nu) \Stepsize{k}{}}{L^2} \norm{\nabla f(X_k)}^2\\
&+ 2\nu (1-\pf)^{1/2} \left (\varConstF \Stepsize{k}{} \norm{\nabla f(X_k)}^2 + \theta \Control{k}{2} \right ) + (1-\pf)^{1/2} \tfrac{4 \gamma (1-\nu) \Stepsize{k}{}}{L^2} \norm{\nabla f(X_k)}^2
\end{align*}
where the inequality follows from $1-\pf \le (1-\pf)^{1/2}$ and $1-\pg = \pf(1-\pg) + (1-\pf)(1-\pg) \le \pf(1-\pg) + (1-\pf)^{1/2}$. Let us choose $\pg \in (0,1]$ so that \eqref{eq:bound_pg} holds which implies 
\begin{align*}
\left ( -\pg \pf \tfrac{ (1-\nu)(1-\gamma^{-1})\Stepsize{k}{}}{L^2}  + \pf(1-\pg) \tfrac{2 \gamma (1-\nu) \Stepsize{k}{}}{L^2} \right ) \norm{\nabla f(X_k)}^2 \le -\pg\pf \tfrac{ (1-\nu)(1-\gamma^{-1})\Stepsize{k}{}}{2 L^2} \norm{\nabla f(X_k)}^2. 
\end{align*}
We have now reduced the number of terms in the conditional expectation 
\begin{align*}
\Exp [ &\Phi_{k+1}-\Phi_k | \pastone] \le -\pg\pf \tfrac{1}{2} (1-\nu) (1-\gamma^{-1}) \left ( \tfrac{\Stepsize{k}{}}{ L^2} \norm{\nabla f(X_k)}^2 + \theta \Control{k}{2} \right )\\
&+ 2\nu (1-\pf)^{1/2} \left (\varConstF \Stepsize{k}{} \norm{\nabla f(X_k)}^2 + \theta \Control{k}{2} \right ) + (1-\pf)^{1/2} \tfrac{4 \gamma (1-\nu) \Stepsize{k}{}}{L^2} \norm{\nabla f(X_k)}^2.
\end{align*}
We choose $\pf \in (0,1]$ large enough, so that $\tfrac{\pg \pf}{\sqrt{1-\pf}}$ satisfies \eqref{eq:bound_product} which implies 
\begin{gather*}
\left ( - \tfrac{\pg\pf(1-\nu)(1-\gamma^{-1}) }{2 L^2} + (1-\pf)^{1/2} \left (2 \nu \varConstF + \tfrac{4 \gamma(1-\nu) }{L^2} \right ) \right ) \Stepsize{k}{} \norm{\nabla f(X_k)}^2 \le - \tfrac{\pg \pf(1-\nu)(1-\gamma^{-1}) \Stepsize{k}{}}{4 L^2} \norm{\nabla f(X_k)}^2\\
-\pg \pf \tfrac{1}{2}(1-\nu)(1-\gamma^{-1}) \theta \Control{k}{2} + 2\nu(1-\pf)^{1/2} \theta \Control{k}{2} \le -\pg \pf \tfrac{1}{4} (1-\nu)(1-\gamma^{-1}) \theta \Control{k}{2}.  
\end{gather*}
The proof is complete. 

\end{proof}

\begin{remark} \label{rmk:constants} \rm{To simplify the expression for the constants we will assume that $\theta = 1/2$ and $\gamma = 2$  which are typical values for these constants. We also assume that without loss of generality $\constG \ge 2$ and $\nu \ge 1/2$. The analysis can be performed for any other values of the above constants - the choices here are for simplicity and illustration. The conditions on $\pg$ and $\pf$ under the above choice of constants will be shown in our results. }\end{remark}

\begin{thm} \label{thm:expected_decrease_simplified} Let
  Assumptions~\ref{assumpt: objective_function} and
  \ref{assumpt:key} hold and chose constants as in Remark~\ref{rmk:constants}. Suppose $\{X_k, G_k, F_k^0, F_k^s,
  \Stepsize{k}{}, \Control{k}{}\}$ is the random process
  generated by Algorithm~\ref{alg:line_search_method}. Then there exists
  probabilities $\pg, \pf > 1/2$ and a constant $\nu \ge 1/2$ such
  that the expected decrease in $\Phi_k$ is
\begin{equation} \label{eq: expect_non_convex_decrease_const} \Exp [\Phi_{k+1}-\Phi_k | \pastone] \le
  - \frac{1}{8192(\constG \stepsize{\max}{} + 1)^2} \left ( \frac{\Stepsize{k}{}}{L^2}
    \norm{\nabla f(X_k)}^2 +  \frac{1}{2} \Control{k}{2} \right ). 
\end{equation}
In particular, the constant $\nu$ and probabilities $\pg, \pf > 1/2$ must satisfy 
\begin{gather}
\frac{\nu}{1-\nu} = 64(\constG \stepsize{\max}{} + 1)^2, \\
\pg \ge \frac{16}{17} \\
\text{and} \qquad \frac{\pg\pf}{\sqrt{1-\pf}} \ge \max \left \{ 1024 \varConstF L^2 (\constG \stepsize{\max}{}+1)^2 + 64, 1024 (\constG \stepsize{\max}{}+1)^2 \right \} . 
\end{gather}
\end{thm}

\begin{proof} We plug in the values for $\gamma$ and $\theta$ and use
  the fact that $\constG \ge 2$ to obtain the expression for
  $\nu/(1-\nu)$ and $\pg$. In order to deduce the expression for
  $\pg\pf/(1-\pf)^{1/2}$, we assume that $\nu/(1-\nu) = 64(\constG
  \stepsize{\max}{} + 1)^2$. Lastly, we suppose $\nu > 1/2$ and $\pg
  \pf \ge 1/2$ and $\tfrac{\nu}{64(\constG \stepsize{\max}{} +1)2} =
  (1-\nu)$. Therefore, we have 
\begin{align*}
\frac{-\pg \pf (1-\nu)(1-\gamma^{-1}) }{4} \le \frac{-(1-\nu)}{64} \le
  \frac{-\nu}{4096(\constG \stepsize{\max}{} +1)^2} \le
  \frac{1}{8192(\constG \stepsize{\max}{} + 1)^2}.
\end{align*}
The result is shown. 
\end{proof}

\begin{cor} \label{cor: summability} Let the same assumptions as Theorem~\ref{thm: expected_decrease} hold. Suppose $\{X_k, G_k, F_k^0, F_k^s,
  \mathcal{A}_k, \Control{k}{}\}$ is the random process generated by
  Algorithm~\ref{alg:line_search_method}. Then we have
\[\sum_{k=0}^\infty {\bf{E}}[ \mathcal{A}_k \norm{\nabla f(X_k)}^2] < \infty.  \] 
\end{cor}

\begin{proof}[Proof of Corollary~\ref{cor: summability}] By taking
  expectations of \eqref{eq: expect_non_convex_decrease} and summing
  up, we deduce
\[ \tilde{C} \sum_{k=0}^\infty {\bf{E}}[\mathcal{A}_k \norm{\nabla f(X_k)}^2]
  \le  \sum_{k=0}^\infty {\bf{E}}[\Phi_k] - {\bf{E}}[\Phi_{k+1}] \le
  \Phi_0 < \infty,\]
where $\tilde{C}$ is the constant in front of the $\mathcal{A}_k
\norm{\nabla f(X_k)}^2$ in \eqref{eq: expect_non_convex_decrease}.
\end{proof}

\subsection{The liminf convergence}\label{sec:liminf} We are ready to prove the
liminf-type of convergence result, i.e. a subsequence of the iterates
drive the gradient of the objective function to zero. The proof
closely follows
\cite{ChenMenickellyScheinberg2014,TR_prob_models_good_fun} for trust
regions; we adapt their proofs to handle line search. 

We set 
\[\bar{\mathcal{A}} = \xi^{-1} \quad \text{where} \quad \xi \ge \max
  \left \{ \frac{\constG +
  L/2 + 2\constF}{1-\theta}, \frac{1}{\stepsize{\max}{}}  \right \}.\]
 For simplicity and without
 loss of generality, we assume that $\Stepsize{0}{} = \gamma^i
 \bar{\mathcal{A}}$ and $\stepsize{\max}{} = \gamma^j
 \bar{\mathcal{A}}$ for integers $i,j > 0$. In this case, for any $k$,
 $\Stepsize{k}{} = \gamma^i \bar{\mathcal{A}}$ for some integer $i$. 

\begin{thm} \label{thm: liminf_type} Let the assumptions of Theorem~\ref{thm: expected_decrease} hold. Then the sequence of random iterates
  generated by Algorithm~\ref{alg:line_search_method}, $\{X_k\}$,
  almost surely satisfy
\[ \liminf_{k \to \infty} \norm{\nabla f(X_k)} = 0. \]
\end{thm}

\begin{proof}  We prove this result by contradiction. With positive
  probability, there
  exists constants $\mathcal{E}(\omega) > 0$ and $K_0(\omega)$ such that $\norm{\nabla f(X_k)} > \mathcal{E}$ for all
  $k \ge K_0$. Because of Corollary~\ref{cor: summability} following Theorem~\ref{thm: expected_decrease}, we have that $\mathcal{A}_k \norm{\nabla
    f(X_k)}^2 \to 0$ a.e. Hence, we have
\[{\bf{Pr}}( \{ \omega \, : \, \norm{\nabla f(X_k)} > \mathcal{E}
  \text{ for all $k \ge K_0$} \text{ and } \lim_{k \to \infty}
  \mathcal{A}_k \norm{\nabla f(X_k)}^2 = 0 \} ) > 0.\]
Let $\{x_k\}, \{\alpha_k\}, \varepsilon, \text{ and } k_0$ be the realizations of
$\{X_k\}$, $\{\mathcal{A}_k\}$, $\mathcal{E}$, and $K_0$, respectively, for which $\norm{\nabla f(x_k)} >
  \varepsilon$ for all $k \ge k_0$ and $\displaystyle \lim_{k \to \infty} \alpha_k
  \norm{\nabla f(x_k)}^2 = 0$. An immediate consequence is that $\alpha_k \to
  0$. Consequently, we deduce that 
\[ 0 < \textbf{Pr}(\{ \omega \, : \, \norm{\nabla f(X_k)} > \mathcal{E}
  \text{ for all $k \ge K_0$} \text{ and } \lim_{k \to \infty}
  \mathcal{A}_k \norm{\nabla f(X_k)}^2 = 0 \})
  \le \textbf{Pr}( \{ \omega \, : \, \lim_{k \to \infty} \mathcal{A}_k = 0\}).\]
We define two new random variables $R_k = \log(\mathcal{A}_k)$ and
$Z_k$ defined by the recursion 
\[Z_{k+1} = \min \left \{ \log(\bar{\mathcal{A}}), 1_{I_k}1_{J_k}  (\log(\gamma) +
    Z_k) + (Z_k-\log(\gamma)) (1-1_{I_k J_k}) \right
  \}  \quad  \text{and} \quad Z_0 = R_0 = \log(\alpha_0). \]
  We observe that $Z_k$ is bounded from below and that $R_k$, by our assumption has a positive probability of diverging to $-\infty$. 
  We establish a contradiction by proving that $R_k \ge Z_k$, which is what we do below. 
  
The sequence of random variables increase by $\log(\gamma)$, unless it
hits the maximum,  with probability $\pf\pg$ and otherwise decreases by
$\log(\gamma)$. Our main argument is to show that $ \displaystyle \{ \omega \, : \,
\lim_{k \to \infty} \mathcal{A}_k =
0\} = \{ \omega \, : \, \lim_{k \to \infty} R_k = -\infty\}$ are null sets. By
construction, the random variables $R_k$ and $Z_k$ are measurable with
respect to the same $\sigma$-algebra namely $\mathcal{F}_{k-1}^{M
  \cdot F}$ for $k \ge 0$. We next show that $R_k \ge Z_k$. The following proceeds
by induction whereas the base case is given by definition. Without
loss of generality assume there exists a $j \in \mathbb{Z}$ such that
$\gamma^j \alpha_0 = \bar{\mathcal{A}}$. Assume the induction
hypothesis, namely, $R_k \ge Z_k$. If $R_k > \log(\bar{\mathcal{A}})$,
then 
\[R_{k+1} \ge \log(\gamma^{-1} \mathcal{A}_k) = R_k - \log(\gamma)
\ge \log(\bar{\mathcal{A}}) \ge Z_k,\] 
where the third inequality follows because the assumption $R_{k+1}$ strictly larger than
$\log ( \bar{\mathcal{A}})$ implies that $R_{k+1} \ge
\log(\bar{\mathcal{A}}) + \log(\gamma)$. Now suppose $R_{k+1} \le
\log ( \bar{\mathcal{A}}) $ and we consider some cases. If $1_{I_k} 1_{J_k}
=1$, then by Lemma~\ref{lem: good_est_good_model_successful} we know
$R_{k+1} = \log ( \mathcal{A}_{k+1} ) = \min \left \{ \log(\alpha_{\max}), R_k +
\log(\gamma) \right \}.$ Suppose $R_{k+1} = \log(\alpha_{\max})$. Then
by definition of $\bar{\mathcal{A}}$ and $Z_{k+1}$, $R_{k+1} \ge
\log ( \bar{\mathcal{A}} ) \ge Z_{k+1}$. On the other hand, suppose $R_{k+1} =
R_k + \log(\gamma)$. Then by the induction hypothesis, we have
$R_{k+1} \ge Z_k + \log(\gamma) \ge \min\{ \bar{\mathcal{A}}, Z_k +
\log(\gamma) \} = Z_{k+1}$. Next, suppose $1_{I_k}1_{J_k} = 0$. It
follows that $Z_{k+1} = Z_k - \log(\gamma) \ge R_k - \log(\gamma) =
\log(\mathcal{A}_k \gamma^{-1}) \le R_{k+1}$. Therefore, we showed
that $R_k \ge Z_k$ for all $k \ge 0$. Moreover, we see that $\{Z_k\}$
is a random walk with a maximum and a drift upward. Therefore, 
\[ 1 = {\bf{Pr}}(\limsup_k Z_k \ge \log ( \bar{\mathcal{A}} ) )  = {\bf{Pr}}(\limsup_k
  R_k \ge \log ( \bar{\mathcal{A}} ) ).\]
However, this contradicts the fact that ${\bf{Pr}}(\omega : \limsup_k
R_k = -\infty) > 0$.
\end{proof}

\subsection{Convergence rates for the nonconvex case} Our primary goal
in this paper is to bound the expected number of steps that the
algorithm takes until $\norm{\nabla f(X_k)} \le \varepsilon$. Define
the stopping time
\[T_\varepsilon = \inf \{ k \ge 0 \, : \, \norm{\nabla f(X_k)} <
\varepsilon\}.\] We show in this section, under the simplified
assumptions on the constant from Theorem~\ref{thm:expected_decrease_simplified}
\[ \Exp [ T_{\varepsilon} ] \le \mathcal{O}(1) \cdot \frac{\pg
    \pf}{2\pg \pf-1} \cdot
   \frac{L^3(\constG \stepsize{\max}{} + 1)^2 \Phi_0}{\varepsilon^2} + 1.\]
Here $\mathcal{O}(1)$ hides universal constants and dependencies on
$\theta$, $\gamma$, $\stepsize{\max}{}$.
We derive this result from Theorem~\ref{thm:renewal_reward_stop_time};
therefore, the remainder of this section is devoted to showing Assumption~\ref{assump:
  random_variable} holds. 
Given Theorem~\ref{thm: expected_decrease}, it is immediate the random variable
$\Phi_k$ defined, as in equation \eqref{eq:phi}, satisfies Assumption~\ref{assump:
  random_variable} (iii) by multiplying both side by the indicator,
$1_{\{T_\varepsilon > k\}}$. In particular, we define the function
$h(\Stepsize{k}{}) = \Stepsize{k}{} \varepsilon^2$
(i.e. $\approx \Stepsize{k}{} \norm{\nabla f(X_k)}^2$)  to obtain from
Theorem~\ref{thm: expected_decrease}
\[\Exp[ 1_{\{T_{\varepsilon} > k\} } (\Phi_{k+1} - \Phi_{k}) | \pastone] \le -\Theta
  h(\Stepsize{k}{}) 1_{\{T_{\varepsilon} > k\} } ,\]
where $\Theta = \tfrac{1}{8192L^2(\constG \stepsize{\max}{} +
  1)^2}$. It remain to show Assumption~\ref{assump: random_variable} (ii) holds.

\begin{lem} Let $\pg$ and $\pf$ be such that
  $\pg \pf\ge 1/2$ then Assumption~\ref{assump: random_variable} (ii)
is satisfied for $W_k = 2(1_{\goodModel \cap \goodFunction} - 1/2)$, $\lambda = \log(\gamma)$,
and $p = \pg \pf$. 
\end{lem}

\begin{proof} By the choice of $\bar{\mathcal{A}}$ we have that $\bar{\mathcal{A}} = \mathcal{A}_0e^{  \lambda \bar j}$
for some $\bar j\in Z$ and $\bar j\leq 0$. It remains to show that 
\begin{align*} 1_{\{T_{\varepsilon} > k\}} \mathcal{A}_{k+1} 
&\ge 1_{\{T_{\varepsilon} > k\}} \min \left \{ \bar{\mathcal{A}},
  \min\{ \stepsize{\max}{}, \gamma \Stepsize{k}{}\} I_kJ_k +
  \gamma^{-1} \Stepsize{k}{}
  (1-1_{\goodModel \cap \goodFunction}) \right \}.
\end{align*}
Suppose $\Stepsize{k}{} > \bar{\mathcal{A}}$. Then $\Stepsize{k}{} \ge
\gamma \bar{\mathcal{A}}$ and hence $\Stepsize{k+1}{} \ge
\bar{\mathcal{A}}$. Now, assume that $\Stepsize{k}{} \le
\bar{\mathcal{A}}$. By definition of $\xi$, we have that 
\[ \Stepsize{k}{} \le \frac{1-\theta}{\constG + L/2 + 2 \constF}. \]
Assume that $1_{\goodModel} = 1$ and $1_{\goodFunction} = 1$. It follows from Lemma~\ref{lem:
  good_est_good_model_successful} that the iteration $k$ is
successful, i.e. $x_{k+1} = x_k + s_k$ and $\stepsize{k+1}{} = \max
\{\stepsize{\max}{}, \gamma \stepsize{k}{}\}$. If $I_k J_k = 0$, then
$\stepsize{k+1}{} \ge \gamma^{-1} \stepsize{k}{}$. 
\end{proof}

Finally substituting the expressions for $h$, $\bar{\mathcal{A}}$, and $\Phi_k$ into the
bound on $\Exp[T_{\varepsilon}]$ from
Theorem~\ref{thm:renewal_reward_stop_time} we obtain the following
complexity result. 

\begin{thm}\label{thm:noncvx_complexity} Under the assumptions in
  Theorem~\ref{thm:expected_decrease_simplified}, suppose the
  probabilities $\pg, \pf > 1/2$ satisfy \begin{gather*}
\pg \ge \frac{16}{17} \qquad 
\text{and} \qquad \frac{\pg\pf}{\sqrt{1-\pf}} \ge \max \left \{ 1024 \varConstF L^2 (\constG \stepsize{\max}{}+1)^2 + 64, 1024 (\constG \stepsize{\max}{}+1)^2 \right \} . 
\end{gather*}
with $\frac{\nu}{1-\nu} = 64(\constG \stepsize{\max}{} + 1)^2$. Then the expected number of iterations that Algorithm~\ref{alg:line_search_method} takes
until $\norm{\nabla f(X_k)}^2 \le \varepsilon$ occurs is bounded as
follows
\begin{align*}
\Exp [ T_{\varepsilon}] \le \frac{\pg\pf}{2\pg\pf-1} \cdot \frac{L^2
  (\constG + L/2 + 2 \constF)(\constG \stepsize{\max}{}
  +1)^2}{\Theta \varepsilon^2} \Phi_0  + 1,
\end{align*}
where $\Theta = 1/16384$ and $\Phi_0 = \nu(f(X_0)-f_{\min}) + (1-\nu)
(1/L^2 \Stepsize{0}{} \norm{\nabla f(X_0)}^2 + 1/2 \Control{0}{2})$.
\end{thm} 

\subsection{Convex Case} 
We now analyze line search
(Algorithm~\ref{alg:line_search_method}) under the setting that the objective
function is convex. 
\begin{assumption} \label{assump: convex_conditons} \rm{Suppose in addition to Assumption \ref{assumpt: objective_function}, $f$ is  convex. Let $x^*$
  denote the global minimizer of $f$ and $f^* = f(x^*)$. We assume
  there exists a constant $D$ such that
\[ \norm{x-x^*} \le D \quad \text{for all $x\in \Omega$,}\]
where $\Omega$ is the set that contains all iteration realizations as stated in Assumption \ref{assumpt: objective_function}. 
Moreover, we assume there exists a $L_f > 0$ such that $\norm{\nabla
  f(x)} \le L_f$ for all $x\in \Omega$. }
\end{assumption}

\begin{remark}
In deterministic optimization it is common to assume that function $f$ has bounded level sets and that all iterates remain within the bounded set 
defined by $f(x)\leq f(x_0)$. For the stochastic case, it is not guaranteed that all iterates remain in the bounded level set because it is possible to take steps that increase the function value. Clearly iterates remain in a (large enough) bounded set with high probability.  Alternatively, if it is known that the optimal solution lies within some bounded set, Algorithm ~\ref{alg:line_search_method} can be simply modified to project iterates onto that set. This modified version for the convex case can be analyzed in almost identical way as is done in Theorem \ref{thm: expected_decrease}. However, for simplicity of the presentation, for the convex case we simply impose Assumption \ref{assump: convex_conditons}. 
\end{remark}

In convex setting, the goal is to bound the expected number of
iterations $T_{\varepsilon}$ of Algorithm~\ref{alg:line_search_method} until 
\[f(x_k)-f^* < \varepsilon.\] 
In deterministic case, the complexity bound is derived by showing that $1/(f(x_k)-f^*)$ has a constant decrease, until the $\varepsilon$-accuracy is reached.  For the randomized line search we follow the same idea, replacing $f(x_k)-f^*$ with
$\Phi_k$ (modified by substituting $f_{\min}$ in \eqref{eq:phi} by $f^*$) and defining the function 
\begin{equation}\label{eq:convex_phi}
\Psi_{k} =  \frac{1}{\nu \varepsilon} -
\frac{1}{\Phi_{k \wedge T_{\varepsilon }}}\footnote{We use $a \wedge b = \min\{a, b\}$.
}.
\end{equation}
We show the random process $\{\Psi_k, \mathcal{A}_k\}$ satisfies Assumption~\ref{assump:
  random_variable} for all $k$. To simplify the argument, we impose an upper bound on $\Control{k}{}$.

\begin{assumption} \label{assump: delta} \rm{ Suppose there exists a constant
    $\control{\max}{}$ such that the random variable $\Control{k}{}
    \le \control{\max}{}$. First, with a simple modification to
    Algorithm~\ref{alg:line_search_method}, we can impose this
    assumption. Second, the dynamics of the algorithm suggest
    $\Control{k}{}$ eventually decreases until it is smaller than any
    $\varepsilon > 0$.}
\end{assumption}

The random variables $\Stepsize{k}{}$  behaves the same as in
the nonconvex setting. We ensure the positivity of the random process
$\{\Psi_k\}$ by incorporating the stopping time $T_{\varepsilon}$
  directly into the definition of $\Psi$; hence the dependency on
  $\varepsilon$ for convergence rates is built directly into the
  function $\Psi$. The main component of this section is proving
Assumption~\ref{assump: random_variable} (iii) holds for this $\Psi_k$,
i.e. an expected improvement occurs.

\begin{thm}\label{thm: convex_expectation} Let Assumptions~\ref{assumpt: objective_function},
  \ref{assumpt:key},
  \ref{assump: convex_conditons}, and \ref{assump: delta}
  hold. Suppose $\{X_k, G_k, F_k^0, F_k^s, \Stepsize{k}{},
  \Control{k}{} \}$ is the random process
  generated by Algorithm~\ref{alg:line_search_method}. Then there exists
  probabilities $\pg$ and $\pf$ and a constant $\nu \in (0,1)$ such that
\[ 1_{\{T_{\varepsilon} > k\}}\Exp [ \Psi_{k+1}-\Psi_k | \pastone] \le
  -\frac{\pg \pf (1-\nu)(1-\gamma^{-1})}{8(\nu D L+
    \tfrac{ (1-\nu)\stepsize{\max}{}L_f}{L} + (1-\nu) \sqrt{\theta}
    \control{\max}{} )^2} \Stepsize{k}{} 1_{\{T_{\varepsilon} > k\}},\]
where $\Psi_k$ is defined in \eqref{eq:convex_phi}.
In particular, the probabilities $\pg$ and $\pf$ and constant $\nu$
satisfy \eqref{eq:bound_nu}, \eqref{eq:bound_pg}, and \eqref{eq:bound_product}
from Theorem~\ref{thm: expected_decrease}. 
\end{thm}

\begin{proof} First, by convexity, we have that
\begin{align*}
\Phi_k = \nu(f(X_k)-f^*) &+ (1-\nu)\Stepsize{k}{} \tfrac{\norm{\nabla f(X_k)}^2}{L^2}
  + (1-\nu) \theta \Control{k}{}\\
&\le \nu
  \ip{\nabla f(X_k), X_k-x^*} + (1-\nu) \stepsize{\max}{} \tfrac{\norm{\nabla
  f(X_k)}^2}{L^2} + (1-\nu) \theta \control{\max}{} \Control{k}{}\\
&\le \left ( \nu DL + \tfrac{ (1-\nu) \alpha_{\max} L_f}{L} +
  (1-\nu)\sqrt{\theta} \control{\max}{}
  \right ) \left ( \tfrac{\norm{\nabla f(X_k)}}{L} + \sqrt{\theta} \Control{k}{} \right ),
\end{align*}
where we used $\norm{\nabla f(X_k)} < L_f$. Without loss of
generality, we assume $\stepsize{\max}{} \le 1$; one may prove the
same result with any stepsize, but for the sake simplicity we will
defer to the standard case when $\stepsize{\max}{} \le 1$. By squaring both sides, we
conclude
\[\frac{\Stepsize{k}{} \Phi_{k}^2}{\tilde{C}} := \frac{\Stepsize{k}{} \Phi_k^2}{2 ( \nu DL + \tfrac{ (1-\nu) \alpha_{\max} L_f}{L} +
  (1-\nu)\sqrt{\theta} \control{\max}{}
   )^2} \le 
  \Stepsize{k}{} \frac{\norm{\nabla f(X_k)}^2}{L^2} + \theta \Control{k}{2},\]
where we used the inequality $(a + b)^2 \le 2(a^2 + b^2)$. From the above inequality combined with  \eqref{eq: expect_non_convex_decrease} 
we have
 \[ \Exp [1_{\{T_\varepsilon > k\}}
(\Phi_{k+1}-\Phi_k)  | \pastone] \le
  \frac{-\pg \pf (1-\nu)(1-\gamma^{-1}) \Stepsize{k}{}
    \Phi_k^2}{4\tilde{C}} \cdot 1_{\{T_{\varepsilon} > k\}}.\]
    Now using the  simple fact that $1_{\{T_\varepsilon > k\}}
(\Phi_{k+1}-\Phi_k) = \Phi_{(k+1) \wedge T_{\varepsilon}} - \Phi_{k
  \wedge T_{\varepsilon}}$
 we can write 
\[ \Exp [\Phi_{(k+1) \wedge T_{\varepsilon}}-\Phi_{k \wedge T_{\varepsilon}} | \pastone] \le
  \frac{-\pg \pf (1-\nu)(1-\gamma^{-1}) \Stepsize{k}{}
    \Phi_k^2}{4\tilde{C}} \cdot 1_{\{T_{\varepsilon} > k\}}.\]
We can then use Jensen's inequality to derive 
\begin{align*}
\Exp \left [ \frac{1}{\Phi_{k \wedge T_{\varepsilon}}} -
  \frac{1}{\Phi_{(k+1) \wedge T_{\varepsilon}}} \big |
  \pastone \right ] &\le \frac{1}{\Phi_{k \wedge T_{\varepsilon}}} -
                      \frac{1}{\Exp [\Phi_{(k+1) \wedge
                      T_{\varepsilon}}|
  \pastone]}  =  \left ( \frac{ \Exp [\Phi_{(k+1) \wedge
                      T_{\varepsilon}} - \Phi_{k\wedge T_{\varepsilon}} |
  \pastone] }{\Phi_{k \wedge T_{\varepsilon}} \Exp [\Phi_{(k+1) \wedge
                      T_{\varepsilon}} |
  \pastone]} \right )\\
&\le
  \frac{-\pg \pf (1-\nu)(1-\gamma^{-1}) \Stepsize{k}{} }{4\tilde{C}}
  \cdot \frac{\Phi_k^2}{\Phi_{k \wedge T_{\varepsilon}} \Exp
  [\Phi_{(k+1) \wedge T_{\varepsilon}} |
  \pastone]} \cdot 1_{\{T_{\varepsilon} > k\} } \\
&\le \frac{-\pg \pf (1-\nu)(1-\gamma^{-1}) \Stepsize{k}{}}{4\tilde{C}}
  \cdot 1_{\{T_{\varepsilon} > k\} }
\end{align*}
where the last inequality follows from 
${\bf{E}}[\Phi_{(k+1) \wedge T_{\varepsilon}} |
\pastone] \le \Phi_{k \wedge T_{\varepsilon}}$. The result follows after
noting that $1_{\{T_{\varepsilon} > k\}}(\Psi_{k+1}-\Psi_k) = \Phi_{k \wedge T_{\varepsilon}}^{-1}
- \Phi_{(k+1) \wedge T_{\varepsilon}}^{-1}$.
\end{proof}

The expected improvement in $\Psi_k$ allows us to use
Theorem~\ref{thm:renewal_reward_stop_time} which directly gives us the
convergence rate.

\begin{thm}\label{thm:cvx_convergence_rate} Let the assumptions of
  Theorem~\ref{thm: convex_expectation} hold with constant $\nu$ and
  probabilities $\pf \pg$ as in Theorem~\ref{thm:
    convex_expectation}. Then the expected number of iterations that
  Algorithm~\ref{alg:line_search_method} takes until $f(X_k)-f^* < \varepsilon$ is bounded
  as follows
\begin{align*}
\Exp [T_\varepsilon] \le \mathcal{O}(1) \cdot \frac{\pg \pf}{2\pg \pf
  -1} \cdot \frac{(\constG \stepsize{\max}{} +1)^2 (\constG + L +
  \constF) (\nu DL + \tfrac{ (1-\nu) \alpha_{\max} L_f}{L} +
  (1-\nu)\sqrt{\theta} \control{\max}{})^2}{\varepsilon } + 1.
\end{align*}
\end{thm}

The bound in Theorem \ref{thm:cvx_convergence_rate} can be further simplified as follows
\begin{align*}
\Exp [T_\varepsilon] \le \mathcal{O}(1) \cdot \frac{\pg \pf}{2\pg \pf
  -1} \left
  ( \frac{L^3\constG^3 (D^2 + L_f^2 + \control{\max}{2})}{\varepsilon} \right ).
\end{align*}

\subsection{Strongly convex case} Lastly, we analyze the stochastic
line search (Algorithm~\ref{alg:line_search_method}) under the setting that the objective
function is strongly convex. As such, we assume the following is now
true of the objective function while dropping Assumption \ref{assump: convex_conditons} 
and the bound on 
$\Control{k}{}$.  

\begin{assumption}\label{assump: strongly_convex_conditons} \rm{ Suppose
  that  in addition to Assumption \ref{assumpt: objective_function} $f$ is $\mu$-strongly convex, namely for all $x, y
  \in \mathbb{R}^n$ the following inequality holds
\[ f(x) \ge f(y) + \nabla f(y)^T(x-y) + \frac{\mu}{2} \norm{x-y}^2. \]
}
\end{assumption}
Our goal, like the convex setting, is to bound the expected number of
iterations $T_{\varepsilon}$ until $f(x)-f^* < \varepsilon$. We show that this bound is of the order of $\log(1/\varepsilon)$, as in the deterministic case. Our proof follows the same technique used in deterministic  which relies on showing that   $\log(f(x_k)-f^*)$ decreases by a constant at each iteration. 
Here, instead of tracking the decrease in $\log(f(x_k)-f^*)$,   we define the function 
\begin{equation}\label{eq:strongly_convex_phi}
\Psi_{k} = \log (\Phi_{k \wedge T_{\varepsilon}}) + \log \left ( \frac{1}{\nu \varepsilon} \right ).
\end{equation}
We show the random process $\{\Psi_k, \mathcal{A}_k\}$ satisfies Assumption~\ref{assump:
  random_variable}. Again, the dynamics of $\Stepsize{k}{}$ do not change
and $\Psi \ge 0$ since we incorporated the stopping time directly into
the definition of $\Psi$. 

\begin{thm}\label{thm: strongly_convex_expectation}
Let Assumptions~\ref{assumpt: objective_function},
  \ref{assumpt:key}, and
  \ref{assump: strongly_convex_conditons} hold. Suppose $\{X_k, G_k, F_k^0, F_k^s, \mathcal{A}_k\}$ is the random process
  generated by Algorithm~\ref{alg:line_search_method}. The expected
  improvement is
\[ 1_{
    \{T_{\varepsilon} > k \}}\Exp [\Psi_{k+1}-\Psi_k | \pastone] \le
  -\frac{\pg \pf (1-\nu)(1-\gamma^{-1})}{4( \tfrac{L^2 \nu}{4\mu} +
    (1-\nu)\stepsize{\max}{} + (1-\nu))} \Stepsize{k}{} \cdot 1_{
    \{T_{\varepsilon} > k \}},\]
where $\Psi_k$ is defined in \eqref{eq:strongly_convex_phi} and the
probabilities $\pg$ and $\pf$ and constant $\nu$ are defined in
Theorem~\ref{thm: expected_decrease}. 
\end{thm}

\begin{proof} By strong convexity, for all $x$, we have $f(x) - f^*
  \le \tfrac{1}{2\mu} \norm{\nabla f(x)}^2$; hence we obtain
\begin{align*}
\Phi_k = \nu ( f(X_k)-f^*) + (1-\nu) &\left ( \Stepsize{k}{} \tfrac{\norm{\nabla
  f(X_k)}^2}{L^2} + \theta \Control{k}{2} \right ) \le \left ( \tfrac{\nu L^2}{2\mu} + (1-\nu) \alpha_{\max}  \right )
  \tfrac{\norm{\nabla f(X_k)}^2}{L^2} + (1-\nu) \theta
  \Control{k}{2}\\
&\le  \left ( \tfrac{\nu L^2}{2\mu} + (1-\nu) \stepsize{\max}{} + (1-\nu)
  \right ) \left ( \tfrac{\norm{\nabla f(X_k)}^2}{L^2} + \theta
  \Control{k}{2} \right ).
\end{align*}
For simplicity of notation we define $ \tilde{C}= \left ( \tfrac{\nu L^2}{2\mu} + (1-\nu) \stepsize{\max}{} + (1-\nu)
  \right )$. 
Also for simplicity and without loss of generality, we assume $\stepsize{\max}{} \le
1$; hence, we conclude \[1_{\{ T_{\varepsilon} > k \}} \Stepsize{k}{}
\Phi_k \le 1_{\{ T_{\varepsilon} > k \}} \tilde{C} \left (
  \tfrac{\Stepsize{k}{} \norm{\nabla f(X_k)}^2}{L^2} + \theta \Control{k}{2}
\right ).\] 
Theorem~\ref{thm: expected_decrease} and the equality $1_{\{T_\varepsilon > k\}}
(\Phi_{k+1}-\Phi_k) = \Phi_{(k+1) \wedge T_{\varepsilon}} - \Phi_{k
  \wedge T_{\varepsilon}}$ give 
\begin{align*}
\Exp [\Phi_{(k+1) \wedge T_{\varepsilon} }-\Phi_{k \wedge T_{\varepsilon}} | \pastone ] &\le -
  \frac{\pg \pf (1-\nu)(1-\gamma^{-1})}{4} \left (
  \tfrac{\Stepsize{k}{} \norm{\nabla f(X_k)}^2}{L^2} + \theta \Control{k}{2}
\right ) 1_{ \{T_{\varepsilon} > k \} }\\
&\le \frac{-\pg \pf (1-\nu)(1-\gamma^{-1}) \Stepsize{k}{}}{ 4\tilde{C}
} \cdot \Phi_k \cdot 1_{ \{T_{\varepsilon} > k\}}
\end{align*}
\begin{equation} \label{eq:strong_cvx_main_inequality}
\Rightarrow \quad \Exp [\Phi_{(k+1) \wedge T_{\varepsilon} } | \pastone]
  \le \left (1 - \frac{\pg \pf (1-\nu)(1-\gamma^{-1}) \Stepsize{k}{}}{
      4 \tilde{C}} \cdot 1_{ \{T_{\varepsilon} > k\} }\right ) \Phi_{k \wedge T_{\varepsilon}}.
\end{equation}
Consequently, using Jensen's inequality, we have the following 
\begin{align*}
\Exp [\log(\Phi_{(k+1) \wedge T_{\varepsilon} }) -\log(\Phi_{k \wedge T_{\varepsilon}})| \pastone] &\le \log \left ( \Exp
                                                  [\Phi_{(k+1) \wedge T_{\varepsilon}} |
                                                  \pastone] \right ) -
                                                  \log (\Phi_{k \wedge
                                                                                                     T_{\varepsilon}})\\
& = \log
                                                  \left ( \frac{\Exp
                                                                                                     [\Phi_{(k+1)
                                                                                                     \wedge
                                                                                                     T_{\varepsilon}}
                                                                                                     |
                                                                                                     \pastone]
                                                                                                     }{\Phi_{k
                                                                                                     \wedge
                                                                                                     T_{\varepsilon}}}
                                                                                                     \right )\\
&\le \log \left (1 - \frac{\pg \pf (1-\nu)(1-\gamma^{-1})
  \Stepsize{k}{}}{ 4\tilde{C}} \cdot 1_{\{T_{\varepsilon} > k\}}\right ),
\end{align*}
where the
last inequality follows by \eqref{eq:strong_cvx_main_inequality}. Because
$\log(1-x) \le -x$ for $x < 1$, we deduce our result. 
\end{proof}


Using the above theorem allows us to use Theorem~\ref{thm:renewal_reward_stop_time}, and after simplifying some constants, 
we have the following complexity bound. 
\begin{thm}\label{thm:cvx_convergence_rate} Let the assumptions of
  Theorem~\ref{thm: convex_expectation} hold with constant $\nu$ and
  probabilities $\pf \pg$ as in Theorem~\ref{thm: strongly_convex_expectation}. Then the expected number of iterations that
  Algorithm~\ref{alg:line_search_method} takes until $f(X_k)-f^* < \varepsilon$ is bounded
  as follows
\begin{align*}
&\Exp [T_\varepsilon] \le \mathcal{O}(1) \cdot \frac{\pg \pf}{2\pg \pf
  -1} \left
  ( (\constG \stepsize{\max}{})^2 (\constG + L +
  \constF) \left ( \frac{L^2}{2\mu} + \stepsize{\max}{}
                                           \right )  \right ) \left (
                                          \log(\Psi_0) + \log \left ( \frac{1}{\varepsilon}
                                          \right ) \right) + 1
\end{align*}
\end{thm}

Simplifying the bound further gives us 
\[\Exp [T_\varepsilon] \le \mathcal{O}(1) \cdot \frac{\pg \pf}{2\pg \pf
  -1} \left ( \frac{L^3 (\constG \stepsize{\max}{})^3 }{\mu} \right )   \left (
                                          \log(\Psi_0) + \log \left ( \frac{1}{\varepsilon}
                                          \right ) \right ).\]



\subsection{General descent, nonconvex case}

In this subsection, we extend the analysis of our line search method
to the general setting where the search direction is any descent
direction $d_k$, which is meant to be a decent direction, but may not be due to stochasticity. For example
$d_k$ can be a direction computed by applying subsampled Newton method \cite{Mahoney}. 
 Algorithm~\ref{alg:line_search_method} is then modified as follows
\begin{itemize}
\item a step is  reliable  when $-\alpha_k g_k^Td_k \ge
  \control{k}{2}$ instead of $\alpha_k \norm{g_k}^2 \ge \control{k}{2}$;
\item the stepsize $s_k = \alpha_k d_k$ (instead of $-\alpha_k g_k$).
\item The sufficient decrease \eqref{eq:
      sufficient_decrease} 
    is replaced with 
\begin{equation} \label{eq: sufficent_decrease_general}
f(x_k + \alpha_k d_k) \le f(x_k) + \alpha_k \theta d_k^T g_k.
\end{equation}
\item  $d_k$  satisfies the following standard
conditions.
\end{itemize}

\begin{assumption} \label{assumpt:general_descent} \rm{Given a gradient estimate $g_k$ we assume the
    following hold for the descent direction $d_k$
\begin{enumerate}
\item There exists a constant $\beta > 0$, such that $d_k$ is a
  descent direction, namely
\[ \frac{d_k^Tg_k}{\norm{d_k}{\norm{g_k}}} \le - \beta, \qquad \text{for all
    $k$.} \]
\item There exist constants $\kappa_1, \kappa_2 > 0$ such that
\[ \kappa_1 \norm{g_k} \le \norm{d_k} \le \kappa_2 \norm{g_k}, \quad \text{for
    all $k$}. \]
\end{enumerate}
} 
\end{assumption}

We now provide simple variants of lemmas derived in
Section~\ref{sec:useful_results}.

\begin{lem}[Variant of Lemma~\ref{lem: bound_gradient}]\label{lem: gen_descent_bound_gradient} Suppose the iterate is
  successful and the descent direction $d_k$ satisfies Assumption \ref{assumpt:general_descent}. Then
\[ \norm{\nabla f(x_{k+1})}^2 \le 2 (L^2 \stepsize{k}{2} \kappa_2^2
  \norm{g_k}^2 + \norm{\nabla f(x_k)}^2). \]
In particular, the inequality holds
\[\tfrac{1}{L^2} \left (\stepsize{k+1}{} \norm{\nabla f(x_{k+1})}^2 -
    \stepsize{k}{} \norm{\nabla f(x_k)}^2 \right ) \le  2 \gamma
  \stepsize{k}{2} \left (\alpha_{\max}^2 \kappa_2^2 \norm{g_k}^2 +
    \tfrac{1}{L^2} \norm{\nabla f(x_k)}^2 \right ) \]
\end{lem}

\begin{proof} An immediate consequence of $L$-smoothness of $f$ is $\norm{\nabla f(x_{k+1})} \le L \alpha_k \norm{d_k} +
  \norm{\nabla f(x_k)}$. The result follows from
  Assumption~\ref{assumpt:general_descent} (ii) then squaring both sides and applying the
  bound, $(a+b)^2 \le 2(a^2 + b^2)$. To obtain the second inequality, we note that in the case $x_k +
  s_k$ is successful, so $\stepsize{k+1}{} = \gamma \stepsize{k}{}$. 
\end{proof}

The analysis for the steepest descent relies on successful iterations
occurring whenever the stepsize is sufficiently small. We provide a
similar result for general descent case.

\begin{lem}[Accurate gradients/estimates $\Rightarrow$
  successful iteration, variant of Lemma~\ref{lem: good_est_good_model_successful} ] \label{lem: good_est_good_model_successful_general_descent}
  Suppose $g_k$ is $\constG$-sufficiently accurate, the descent
  direction $d_k$ satisfies Assumption~\ref{assumpt:general_descent}, and $\{f_k^0, f_k^s\}$ are $\constF$-accurate estimates. If 
\[\stepsize{k}{} \le  \frac{\beta (1-\theta)}{\constG + \frac{L \kappa_2}{2}
      + \frac{2\constF}{\kappa_1}}\]
then the trial step $x_k + s_k$ is successful. In particular, this means $f_k^s \le f_k^0 + \theta \stepsize{k}{} g_k^Td_k.$
\end{lem}

\begin{proof} The $L$-smoothness of $f$ and the $\constG$-sufficiently
  accurate gradient immediately yield 
\begin{align*}
f(x_k + s_k) &\le f(x_k) + \stepsize{k}{} (\nabla f(x_k)-g_k)^Td_k + \stepsize{k}{}
  g_k^Td_k + \tfrac{L \stepsize{k}{2}}{2} \norm{d_k}^2\\
&\le f(x_k) + \constG \stepsize{k}{2} \norm{d_k} \norm{g_k} +
  \stepsize{k}{} g_k^T d_k
+ \tfrac{L \stepsize{k}{2}}{2} \norm{d_k}^2.
\end{align*}
Since the estimates are $\constF$-accurate, we obtain
\begin{align*}
f_k^s - \constF \stepsize{k}{2} \norm{g_k}^2 &\le f(x_k + s_k) - f_k^s + f_k^s\\
  &\le f(x_k) - f_k^0 + f_k^0 +\constG\stepsize{k}{2} \norm{d_k}
    \norm{g_k} + \stepsize{k}{} g_k^T d_k + \tfrac{L \stepsize{k}{2}}{2} \norm{d_k}^2 \\
&\le f_k^0 + \constF \stepsize{k}{2} \norm{g_k}^2 + \constG
  \stepsize{k}{2} \norm{d_k} \norm{g_k} + \stepsize{k}{} g_k^Td_k + \tfrac{L \stepsize{k}{2}}{2} \norm{d_k}^2.
\end{align*}
The above inequality with Assumption~\ref{assumpt:general_descent} implies 
\begin{align*}
f_k^s-f_k^0 &\le \alpha_k^2 \left ( \frac{2\constF }{\kappa_1} +
  \constG + \frac{L \kappa_2}{2} \right ) \norm{g_k} \norm{d_k} +
  \alpha_k g_k^T d_k\\
&\le \frac{-\alpha_k^2}{\beta} \left ( \frac{2\constF }{\kappa_1} +
  \constG + \frac{L \kappa_2}{2} \right ) g_k^Td_k +
  \alpha_k g_k^T d_k.
\end{align*}
The result follows by noting $f_k^s \le f_k^0 + \stepsize{k}{} g_k^Td_k \left ( 1- \frac{\stepsize{k}{}}{\beta} \left (
  \constG + \frac{L\kappa_2}{2} + \frac{2\constF}{\kappa_1}
  \right ) \right )$.
\end{proof}

As in the steepest descent case, we can use the same function $\Phi_k$
as defined in \eqref{eq:phi}. Using the sufficient decrease condition \eqref{eq: sufficent_decrease_general}
and Assumption \ref{assumpt:general_descent} on $d_k$, a successful iterate yields a decrease of 
\[ f(x_k + \stepsize{k}{} d_k) \le -\theta \stepsize{k}{} \kappa_1
    \beta \norm{g_k}^2. \] Hence, we can derive, as in the steepest
  descent scenario, an expected decrease in $\Phi_k$. 

\begin{thm} \label{thm: gen_descent_expected_decrease} Let
  Assumptions~\ref{assumpt: objective_function},
  \ref{assumpt:key}, and \ref{assumpt:general_descent} hold. Suppose
  $\{X_k, D_k, G_k, F_k^0, F_k^s,
  \Stepsize{k}{}, \Control{k}{}\}$ is the random process
  generated by Algorithm~\ref{alg:line_search_method}. Then there exist
  probabilities $\pg, \pf > 1/2$ and a constant $\nu \in (0,1)$ such
  that the expected decrease in $\Phi_k$ is
\begin{equation} \label{eq: expect_non_convex_decrease} \Exp [\Phi_{k+1}-\Phi_k | \pastone] \le
  - \frac{\pg \pf (1-\nu) (1-\gamma^{-1})}{4} \left ( \frac{\Stepsize{k}{}}{L^2}
    \norm{\nabla f(X_k)}^2 + \theta \Control{k}{2} \right ). 
\end{equation}
In particular, the constant $\nu$ and probabilities $\pg, \pf > 1/2$ satisfy 
\begin{gather}
\frac{\nu}{1-\nu} \ge \max \left \{ \frac{32 \gamma
    \stepsize{\max}{2} \kappa_2^2}{\theta \kappa_1 \beta}, 16(\gamma-1), \frac{16 \gamma
    (\constG \stepsize{\max}{} + 1)^2}{\theta \kappa_1 \beta} \right \}, \label{eq:bound_nu}\\
\pg \ge \frac{2 \gamma}{1/2(1-\gamma^{-1}) +2 \gamma} \label{eq:bound_pg}\\
\text{and} \qquad \frac{\pg\pf}{\sqrt{1-\pf}} \ge \max \left \{
  \frac{ 8L^2\nu\varConstF + 16 \gamma(1-\nu)}{
    (1-\nu)(1-\gamma^{-1})}, \frac{8\nu}{(1-\nu)(1-\gamma^{-1})} \right \} \label{eq:bound_product}. 
\end{gather}
\end{thm}

\begin{proof} Using Assumption~\ref{assumpt:general_descent} on the descent direction $d_k$
  when an iterate is successful, we see 
\[\success (f(X_k + \Stepsize{k}{} D_k) - f(X_k)) \le -\success
  \theta \Stepsize{k}{} \kappa_1 \beta \norm{G_k}^2. \]
\end{proof}
Hence, we may replace $\theta$ in the proof of Theorem~\ref{thm: expected_decrease} with
$\theta \kappa_1 \beta$. The only other change to the proof and the resulting constants  lies in the replacement of 
 of  Lemma~\ref{lem: bound_gradient} by Lemma~\ref{lem:
  gen_descent_bound_gradient}. This implies a change in the  choice of
$\nu$ in equations \eqref{eq:case_1_1_nu}. In particular, we choose
$\nu$ to now satisfy
\begin{equation} \label{eq:case_1_1_nu} \begin{aligned}
-\frac{\nu \theta \kappa_1 \beta \Stepsize{k}{}}{8} \norm{G_k}^2 &+ (1-\nu)2 \gamma
\Stepsize{k}{} \stepsize{\max}{2} \kappa_2^2 \norm{G_k}^2 \le
-\frac{\nu \theta \kappa_1 \beta
  \Stepsize{k}{}}{16} \norm{G_k}^2,\\
-\frac{\nu \theta \kappa_1 \beta \Stepsize{k}{}}{4L^2(\constG \stepsize{\max}{}
  +1)^2} \norm{\nabla f(X_k)}^2 + &(1-\nu) \frac{2 \gamma \Stepsize{k}{}}{L^2} 
\norm{\nabla f(X_k)}^2 \le - \frac{\nu \theta \kappa_1 \beta
  \Stepsize{k}{}}{8L^2(\constG \stepsize{\max}{} +1)^2} \norm{\nabla f(X_k)}^2,\\
\text{and} \qquad -\frac{\nu \theta}{8} \Control{k}{2} &+ (1-\nu) (\gamma-1) \theta
\Control{k}{2} \le - \frac{\nu \theta}{16} \Control{k}{2}.  
\end{aligned}
\end{equation}

Using Lemma~\ref{lem: good_est_good_model_successful_general_descent}, we can set 
\[\bar{\mathcal{A}} = \frac{\beta (1-\theta)}{\kappa_g + \frac{L
      \kappa_2}{2} + \frac{2 \constF}{\kappa_1}}\]
      and apply Theorem~\ref{thm:renewal_reward_stop_time},

\begin{thm}\label{thm:noncvx_complexity} Under the assumptions in
  Theorem~\ref{thm: gen_descent_expected_decrease}, 
  and constants chosen in Remark \ref{rmk:constants}, suppose the
  probabilities $\pg, \pf > 1/2$ satisfy \begin{gather*}
\pg \ge \frac{16}{17} \quad 
\text{and}\\
 \frac{\pg\pf}{\sqrt{1-\pf}} \ge \max \left \{
  \frac{1024 \varConstF L^2 ( \max\{\constG, 2\kappa_2\} \stepsize{\max}{}+1)^2}{\kappa_1
    \beta} + 64,  \frac{1024 (\max\{ \constG, 2\kappa_2\} \stepsize{\max}{}+1)^2}{\kappa_1
  \beta} \right \}
\end{gather*}
with $\frac{\nu}{1-\nu} = \tfrac{64(\max\{\constG , 2\kappa_2\}
  \stepsize{\max}{} + 1)^2}{\kappa_1 \beta}$. Then the expected number of iterations that Algorithm~\ref{alg:line_search_method} takes
until $\norm{\nabla f(X_k)}^2 \le \varepsilon$ occurs is bounded as
follows
\begin{align*}
\Exp [ T_{\varepsilon}] \le \frac{\pg\pf}{2\pg\pf-1} \cdot \frac{L^3
  \constG^3 \kappa_2^3 \Phi_0}{\kappa_1^2 \beta^2 } \cdot \frac{1}{\varepsilon^2} + 1,
\end{align*}
where $\Phi_0 = \nu(f(X_0)-f_{\min}) + (1-\nu)
(1/L^2 \Stepsize{0}{} \norm{\nabla f(X_0)}^2 + 1/2 \Control{0}{2})$.
\end{thm}

\section{Conclusions}
We have used a general framework based on analysis of stochastic processes proposed in \cite{TR_prob_model} with the purpose of 
analyzing convergence rates of stochastic optimization methods.  In \cite{TR_prob_model} the framework is used to analyze stochastic trust region method, while in this paper we were able to use the same framework to develop and analyze stochastic back-tracking line search method. 
Our method is the first implementable stochastic line-search method that has theoretical convergence rate guarantees. In particular, the the accuracy of gradient and function estimates is chosen dynamically and the requirements of this accuracy are all stated in terms of knowable quantities. We establish complexity results for convex, strongly convex and general nonconvex, smooth stochastic functions.

\bibliographystyle{plain}
\bibliography{bibliography}

\end{document}